\documentclass[12pt,oneside,reqno]{amsart}

\usepackage{amsmath,amssymb,amsthm,textcomp}
\usepackage{dsfont}
\usepackage{amsfonts,graphicx}
\usepackage[mathscr]{eucal}
\usepackage{color}
\usepackage{csquotes}
\usepackage{relsize}
\usepackage{hyperref}
\usepackage[backend=bibtex,%
firstinits=true,%
doi=false,%
isbn=true,%
url=false,%
maxnames=99]{biblatex}%

\AtEveryBibitem{\clearfield{issn}}
\AtEveryCitekey{\clearfield{issn}}
\numberwithin{equation}{section}
\DeclareNameAlias{sortname}{last-first}
\theoremstyle{definition}
\usepackage{mathtools}
\numberwithin{equation}{section}

\newcommand{\ncom}{\newcommand}

\ncom{\beq}{\begin{equation}}
	\ncom{\eeq}{\end{equation}}
\ncom{\bea}{\begin{eqnarray*}}
	\ncom{\eea}{\end{eqnarray*}}
\ncom{\beqa}{\begin{eqnarray}}
	\ncom{\eeqa}{\end{eqnarray}}
\ncom{\nno}{\nonumber}
\ncom{\non}{\nonumber}
\ncom{\ds}{\displaystyle}
\ncom{\half}{\frac{1}{2}}
\ncom{\mbx}{\makebox{.25cm}}
\ncom{\hs}{\mbox{\hspace{.25cm}}}
\ncom{\rar}{\rightarrow}
\ncom{\Rar}{\Rightarrow}
\ncom{\noin}{\noindent}
\ncom{\bc}{\begin{center}}
	\ncom{\ec}{\end{center}}
\ncom{\sz}{\scriptsize}
\ncom{\rf}{\ref}
\ncom{\s}{\sqrt{2}}
\ncom{\sgm}{\sigma}
\ncom{\Sgm}{\Sigma}
\ncom{\psgm}{\sigma^{\prime}}
\ncom{\dt}{\delta}
\ncom{\Dt}{\Delta}
\ncom{\lmd}{\lambda}
\ncom{\Lmd}{\Lambda}
\ncom{\Th}{\Theta}
\ncom{\e}{\eta}
\ncom{\eps}{\epsilon}
\ncom{\pcc}{\stackrel{P}{>}}
\ncom{\lp}{\stackrel{L_{p}}{>}}
\ncom{\dist}{{\rm\,dist}}
\ncom{\sspan}{{\rm\,span}}
\ncom{\re}{{\rm Re\,}}
\ncom{\im}{{\rm Im\,}}
\ncom{\sgn}{{\rm sgn\,}}
\ncom{\ba}{\begin{array}}
	\ncom{\ea}{\end{array}}
\ncom{\hone}{\mbox{\hspace{1em}}}
\ncom{\htwo}{\mbox{\hspace{2em}}}
\ncom{\hthree}{\mbox{\hspace{3em}}}
\ncom{\hfour}{\mbox{\hspace{4em}}}
\ncom{\vone}{\vskip 2ex}
\ncom{\vtwo}{\vskip 4ex}
\ncom{\vonee}{\vskip 1.5ex}
\ncom{\vthree}{\vskip 6ex}
\ncom{\vfour}{\vspace*{8ex}}
\ncom{\norm}{\|\;\;\|}
\ncom{\integ}[4]{\int_{#1}^{#2}\,{#3}\,d{#4}}
\ncom{\vspan}[1]{{{\rm\,span}\{ #1 \}}}
\ncom{\dm}[1]{ {\displaystyle{#1} } }
\ncom{\ri}[1]{{#1} \index{#1}}

\newtheorem{theorem}{\bf Theorem}[section]
\newtheorem{remark}{\bf Remark}[section]
\newtheorem{proposition}{Proposition}[section]
\newtheorem{lemma}{Lemma}[section]
\newtheorem{corollary}{Corollary}[section]

\newtheoremstyle
{remarkstyle}
{}
{11pt}
{}
{}
{\bfseries}
{:}
{     }
{\thmname{#1} \thmnumber{#2} }

\theoremstyle{remarkstyle}

\def\eps{\varepsilon}

\begin{document}
	\title{On Time-Changed Birth-Death Processes with Catastrophes}
	\author[Kuldeep Kumar Kataria]{Kuldeep Kumar Kataria}
	\address{Kuldeep Kumar Kataria, Department of Mathematics, Indian Institute of Technology Bhilai, Durg 491002, India.}
	\email{kuldeepk@iitbhilai.ac.in}
	\author[Rohini Bhagwanrao Pote]{Rohini Bhagwanrao Pote}
	\address{Rohini Bhagwanrao pote, Department of Mathematics, Indian Institute of Technology Bhilai, Durg 491002, India.}
	\email{rohinib@iitbhilai.ac.in}
	\date{\today}
	
\begin{abstract}
We study two time-changed variants of the birth–death process with catastrophe where the time-changing components are the first hitting times of the stable subordinator and the tempered stable subordinator. For both the processes, we derive the governing system of fractional differential equations for their state probabilities. The Laplace transforms of these state probabilities are obtained in terms of those of the corresponding time-changed birth–death processes without catastrophes. We obtain the distribution of catastrophe occurrence times as well as the sojourn times within non-zero states. We study  distributional properties of the first visit time to state zero in a particular case. Also, the first occurrence time of an effective catastrophe is studied. Moreover, we study the time-changed linear birth-death processes with catastrophes, derive the explicit expressions for its state probabilities, expectation and variance. For a specific case, we compare the expectation plots across different parameter values and provide an algorithm for simulating sample paths with illustrative plots.

\end{abstract}
\maketitle 

\section{Introduction}\label{sec1}
A birth–death process is a continuous-time and discrete state space Markov process that serves as a model for population growth over time. In a birth-death process with catastrophe (BDPC), there is a possibility of transition from any state to state zero via catastrophe. The extinction time of a birth-death process with catastrophes was studied by Brockwell (1986) whereas Chao and Zheng (2003) studied an immigration birth–death process with total catastrophes, deriving both its transient and equilibrium behaviour. Di Crescenzo et al. (2003, 2018) studied the $M/M/1$ queue with catastrophes, and the first occurrence time of an effective catastrophe was studied by Di Crescenzo et al. (2008).

Let $\{\mathcal{N}^{\nu}(t)\}_{t\geq0}$ be a BDPC with state space $I=\{0,1,2,\dots\}$. For $t\geq0$, let $p_{m,n}^{\nu}(t)=\mathrm{Pr}\{\mathcal{N}^{\nu}(t)=n|\mathcal{N}^{\nu}(0)=m\}$, $m,n\in I$ denote its state probabilities. The system of differential equations that governs its state probabilities is as follows (see Di Crescenzo {\it et al.} (2008)):
\begin{align}\label{DE1}
	\left.
	\begin{aligned}
		\frac{\mathrm{d}}{\mathrm{d}t}p_{m,0}^{\nu}(t)&=-(\lambda_{0}+\nu)p_{m,0}^{\nu}(t)+\mu_{1}p_{m,1}^{\nu}(t)+\nu,\\
		\frac{\mathrm{d}}{\mathrm{d}t}p_{m,n}^{\nu}(t)&=-(\lambda_{n}+\mu_{n}+\nu)p_{m,n}^{\nu}(t)+\lambda_{n-1}p_{m,n-1}^{\nu}(t)+\mu_{n+1}p_{m,n+1}^{\nu}(t),\,n=1,2,\dots
	\end{aligned}
	\right\}
\end{align}
with initial condition $p_{m,n}^{\nu}(0)=\delta_{m,n}=\begin{cases}
	1,\,n=m,\\
	0,\,\text{otherwise}.
\end{cases}$

Here, $\lambda_{n}$'s are birth rates, $\mu_{n}$'s are death rates and $\nu$ is the rate of occurrence of catastrophe at any state in $I$.

Also, let $\{\mathcal{N}(t)\}_{t\geq0}$ be the corresponding birth-death process without catastrophe (BDP). For $\nu=0$, $\{\mathcal{N}^{\nu}(t)\}_{t\geq0}$ reduces to $\{\mathcal{N}(t)\}_{t\geq0}$. Its state probabilities $p_{m,n}(t)=\mathrm{Pr}\{\mathcal{N}(t)=n|\mathcal{N}(0)=m\}$, $m,n\in I$, $t\geq0$ satisfy the system of differential equations given in \eqref{DE1} with $\nu=0$, and subject to the initial condition $p_{m,n}(0)=\delta_{m,n}$.

A time-changed stochastic process is defined by the composition of a random process with a non-decreasing random process. Typically, these two processes are assumed to be independent. Over the last two decades, these processes have been studied by many authors due to their applications in various fields where time does not occur in a deterministic manner. For example, Beghin and Orsingher (2009) and Meerschaert et al. (2011) studied the fractional Poisson process. The space-fractional Poisson process was introduced by Orsingher and Polito (2012), while time-changed birth–death processes were studied by Orsingher and Polito (2011), Kataria and Vishwakarma (2025a), {\it etc}. Some recent works in this direction are fractional queues with catastrophes and their transient behaviour (see Ascione {\it et al.} (2018)), tempered fractional Poisson processes and fractional equations with Z-transform (see Gupta et al. (2020)), the tempered Erlang queue with multiple arrivals (see Dhillon and Kataria (2025)), Erlang queue with multiple arrivals and its time-changed variant (see Pote and Kataria (2025)), {\it etc}.

In this paper, we introduce and study two time-changed variants of the BDPC where the time-changing components are the inverse stable subordinator and the inverse tempered stable subordinator. Also, we study their linear versions. The paper is organized as follows:

In Section \ref{SecPreliminaries}, we give some definitions and known results on the Mittag-Leffler function, fractional derivatives, stable subordinator and its inverse, and results related to their Laplace transforms.

In Section \ref{SecTime-ChngBDPC}, we define time-changed birth-death process with catastrophe (time-changed BDPC) and time-changed birth-death process without catastrophe (time-changed BDP). The  time-changing component considered is the first hitting time of a stable subordinator known as the inverse stable subordinator. We derive the governing system of fractional differential equations for its state probabilities, and obtain the Laplace transform of state probabilities of time-changed BDPC in terms of that of time-changed BDP. We obtain the distribution of occurrence times of the catastrophes, and that of sojourn times in non-zero states for the time-changed BDPC. 
For the time-changed BDPC with zero as an absorbing state, it is shown that the first visit time to state zero is equal in distribution to the minimum of first visit time to state zero for the time-changed BDP and a Mittag-Leffler random variable. Moreover, we derive the Laplace transform of its probability density function (pdf). The effective catastrophe is discussed, and the Laplace transform of the pdf of its first occurrence time is obtained.

In Section \ref{sec4}, we study a particular case of the time-changed BDPC, known as the time-changed linear birth-death process with catastrophe, in short, the time-changed LBDPC. This is obtained by setting birth and death rates as $\lambda_{n}=n\lambda$ and $\mu_{n}=n\mu$ for any state $n\in I$ in the time-changed BDPC. We derive its extinction probability, state probabilities, expectation and variance. We compare the plots of the expectation for different parameter values and provide a table of numerical examples for both expectation and variance. Also, we give an algorithm to simulate the time-changed LBDPC and provide the sample path simulation plots for it.

In Section \ref{SecTempBDPC}, we study another time-changed variant of the BDPC, namely, the tempered BDPC. It is obtained by time-changing the BDPC with the first hitting time of the tempered stable subordinator. We derive the governing system of fractional differential equations for its state probabilities, the Laplace transform of these state probabilities, and the distributions of the occurrence times of catastrophes and the sojourn times. For the case in which zero is an absorbing state, we show the equality in distribution for the first visit time of state zero of the tempered BDPC to the minimum of the first visit time to state zero of the tempered BDP and a tempered Mittag-Leffler random variable. We derive the Laplace transform of its pdf, expectation and variance. Also, we study the first occurrence time of an effective catastrophe, obtain the Laplace transform of its pdf, expectation and variance. Moreover, we define the tempered linear birth-death process with catastrophe, that is, the tempered LBDPC. We state the results which includes its extinctions probability, state probabilities, expectation and variance.

\section{Preliminaries}\label{SecPreliminaries}
Here, we collect some known definitions and results related to Mittag-Leffler function, fractional derivatives, subordinators and their inverse.


\subsection{Mittag-Leffler function}\label{SubSecMittagLeffler}
Let $\mathbb{R}$ denote the set of real numbers. For $\alpha>0$, $\beta>0$ and $\gamma>0$, the three-parameter Mittag-Leffler function is defined as follows (see Kilbas {\it et al.} (2006)):
\begin{equation*}\label{Mittag12}
	E_{\alpha,\beta}^{\gamma}(t)\coloneqq\sum_{r=0}^{\infty}\frac{\Gamma(\gamma+r)t^{r}}{r!\Gamma(\gamma)\Gamma(\alpha r+\beta)},\,t\in\mathbb{R}.
\end{equation*}
For $\gamma=1$, it reduces to the two-parameter Mittag-Leffler function which is denoted by $E_{\alpha,\beta}(\cdot)$. Further, for $\gamma=\beta=1$, it is known as the one-parameter Mittag-Leffler function, denoted by $E_{\alpha}(\cdot)$.

The Laplace transform of auxiliary function of three-parameter Mittag-Leffler function is given by
\begin{equation}\label{ltm12}
	\mathbb{L}(t^{\beta-1}E_{\alpha,\beta}^{\gamma}(\omega t^{\alpha}))(z)=\frac{z^{\alpha\gamma-\beta}}{(z^{\alpha}-\omega)^{\gamma}},\, \omega\in\mathbb{R},\,z>|\omega|^{1/\alpha},\,t>0.
\end{equation}

The distribution function of Mittag-Leffler random variable $T$ with parameters $\alpha$ and $\lambda$ is
\begin{equation*}
	F_{T}(t)\coloneqq\mathrm{Pr}\{T\leq t\}=1-E_{\alpha}(-\lambda t^{\alpha}).
\end{equation*}


\subsection{Fractional derivatives}\label{SubSecFractionalDerivative}
Let $\theta>0$ and $0<\alpha<1$. 
The Riemann-Liouville fractional derivative of a function $g(\cdot)$ is defined as follows (see  Kilbas {\it et al.} (2006)):
\begin{equation*}
	D^{\alpha}_{t}g(t)\coloneqq\frac{1}{\Gamma(1-\alpha)}\frac{\mathrm{d}}{\mathrm{d}t}\int_{0}^{t}\frac{g(y)}{(t-y)^{\alpha}}\mathrm{d}y.
\end{equation*}
Also, the Riemann-Liouville tempered fractional derivative is defined as (see Alrawashdeh {\it et al.} (2017), Eq. (2.6)):
\begin{equation*}
	D^{\theta,\alpha}_{t}g(t)\coloneqq e^{-\theta t}D^{\alpha}_{t}(e^{\theta t}g(t))-\theta^{\alpha}g(t).
\end{equation*}

The Caputo fractional derivative of a function $g(\cdot)$ is defined as follows (see Kilbas {\it{et al.}} (2006)):
\begin{equation*}\label{caputo}
	\frac{\mathrm{d}^{\alpha}}{\mathrm{d}t^{\alpha}}g(t)=
	\begin{cases}
		\frac{1}{\Gamma(1-\alpha)}\int_{0}^{t}(t-y)^{-\alpha}g'(y)\mathrm{d}y,\, 0<\alpha<1, \\
		g'(t), \, \alpha=1
	\end{cases}
\end{equation*}
whose Laplace transform is given by
\begin{equation}\label{caputolp}
	\mathbb{L}\Big(\frac{\mathrm{d}^{\alpha}}{\mathrm{d}t^{\alpha}}g(t)\Big)(z)=z^{\alpha}\Tilde{g}(z)-z^{\alpha-1}g(0),
\end{equation} 
where $\tilde{g}(z)$ denotes the Laplace transform of $g(t)$. 

Moreover, the definition of Caputo tempered fractional derivative is given by (see Alrawashdeh {\it et al.} (2017), Eq. (2.8))
\begin{equation*}\label{Caputotempered}
	\frac{\mathrm{d}^{\theta,\alpha}}{\mathrm{d}t^{\theta,\alpha}}g(t)=D^{\theta,\alpha}_{t}g(t)-\frac{f(0)}{\Gamma(1-\alpha)}\int_{t}^{\infty}\alpha e^{-\theta y}y^{-\alpha-1}\mathrm{d}y,
\end{equation*}
and its Laplace transform is 
\begin{equation}\label{CaputotemperedLT}
	\mathbb{L}\Big(\frac{\mathrm{d}^{\theta,\alpha}}{\mathrm{d}t^{\theta,\alpha}}g(t)\Big)(z)=\phi_{\theta,\alpha}(z)\Tilde{g}(z)-\frac{\phi_{\theta,\alpha}(z)}{z}g(0),\,z>0,
\end{equation}
where $\phi_{\theta,\alpha}(z)=(z+\theta)^{\alpha}-\theta^{\alpha}$.

\subsection{Stable subordinators and their inverse}\label{SubSecSubordinator}
Let $\{D_{\alpha}(t)\}_{t\geq0}$, $0<\alpha<1$ be a one dimensional L\'evy process whose Laplace transform is given by $\mathbb{E}(e^{-zD_{\alpha}(t)})=e^{-tz^{\alpha}}$, $z>0$. It is called an $\alpha$-stable subordinator (see Applebaum (2009)). Let $\{Y_{\alpha}(t)\}_{t\geq0}$ be its first hitting time, also known as the inverse $\alpha$-stable subordinator, that is, $Y_{\alpha}(t)=\inf\{s\geq0:D_{\alpha}(s)>t\}$. Its probability density function (pdf) has the following Laplace transform (see Meerschaert and Straka (2013)):
\begin{equation}\label{ltLnu12}
	\mathbb{L}(\mathrm{Pr}\{Y_{\alpha}(t)\in\mathrm{d}y\})(z)=z^{\alpha-1}e^{-yz^{\alpha}}\mathrm{d}y, \, z>0.
\end{equation}

A one-dimensional L\'evy process
$\{D_{\theta,\alpha}(t)\}_{t\geq 0}$ whose Laplace transform is $\mathbb{E}(e^{-zD_{\theta,\alpha}(t)})=e^{-t\phi_{\theta,\alpha}(z)}$, $z>0$ is known as the tempered stable subordinator. Here, $\theta>0$ is its tempering parameter and $0<\alpha<1$ its stability index such that its Laplace exponent is given by $\phi_{\theta,\alpha}(z)=(z+\theta)^{\alpha}-\theta^{\alpha}$. Its first hitting time process $\{Y_{\theta,\alpha}(t)\}_{t\geq0}$, that is, $Y_{\theta,\alpha}(t)=\inf\{s\geq0:D_{\theta,\alpha}(s)>t\}$ is called as the inverse tempered stable subordinator. Its pdf has the following Laplace transform (see Meerschaert \textit{et al.} (2013)):
\begin{equation}\label{temperedPDFLT}
	\mathbb{L}(\mathrm{Pr}\{Y_{\theta,\alpha}(t)\in\mathrm{d}y\})(z)=\frac{\phi_{\theta,\alpha}(z)}{z}e^{-y\phi_{\theta,\alpha}(z)}\mathrm{d}y, \, z>0.
\end{equation}

\section{Time-Changed Birth-Death Process with Catastrophe}\label{SecTime-ChngBDPC}
In this section, we introduce a time-changed birth-death process with catastrophe where the time-changing component is the inverse stable subordinator. We call it the time-changed BDPC. 

Let $\{\mathcal{N}^{\nu}(t)\}_{t\geq0}$ be a BDPC with state space $I=\{0,1,2,\dots\}$ such that the state dependent positive birth rates are $\lambda_{n}$'s for $n=1,2,\dots$ with $\lambda_{0}\geq0$, the positive death rates are $\mu_{n}$'s for $n=1,2,3,\dots$, and at any state catastrophe happens with rate $\nu\geq0$. Also, let $\{Y_{\alpha}(t)\}_{t\geq0}$ be an inverse stable subordinator with $Y_{\alpha}(0)=0$ such that it is independent of the BDPC $\{\mathcal{N}^{\nu}(t)\}_{t\geq0}$. Then, we define the time-changed BDPC $\{\mathcal{N}^{\alpha,\nu}(t)\}_{t\geq0}$, $0<\alpha\leq1$, $\nu\geq0$ as follows: 
\begin{equation*}
	\mathcal{N}^{\alpha,\nu}(t)\coloneqq\mathcal{N^{\nu}}(Y_{\alpha}(t)),\,0<\alpha<1,
\end{equation*}
with $\mathcal{N}^{1,\nu}(t)=\mathcal{N}^{\nu}(t)$, $t\geq0$. 

Also, the time-changed birth-death process without catastrophe, that is, the time-changed BDP $\{\mathcal{N}^{\alpha,0}(t)\}_{t\geq0}$, $0<\alpha\leq1$ is defined as
\begin{equation*}
	\mathcal{N}^{\alpha,0}(t)\coloneqq\mathcal{N}(Y_{\alpha}(t)),\,0<\alpha<1,
\end{equation*}
with $\mathcal{N}^{1,0}(t)=\mathcal{N}(t)$, $t\geq0$. For $\nu=0$, the time-changed BDPC reduces to the time-changed BDP.

For $m,n\in I$, let
\begin{equation*}
	p_{m,n}^{\alpha,\nu}(t)=\mathrm{Pr}\{ \mathcal{N}^{\alpha,\nu}(t)=n|\mathcal{N}^{\alpha,\nu}(0)=m\}
\end{equation*}
and
\begin{equation*}
	p_{m,n}^{\alpha,0}(t)=\mathrm{Pr}\{\mathcal{ N}^{\alpha,0}(t)=n|\mathcal{N}^{\alpha,0}(0)=m\}
\end{equation*} 
be the state probabilities of $\{\mathcal{N}^{\alpha,\nu}(t)\}_{t\geq0}$ and $\{\mathcal{N}^{\alpha,0}(t)\}_{t\geq0}$, respectively.



\begin{theorem}\label{thm1}
The state probabilities of time-changed BDPC $\{\mathcal{N}^{\alpha,\nu}(t)\}_{t\geq0}$ solve the following system of fractional differential equations:
\begin{align*}\label{DE2}
	\frac{\mathrm{d}^{\alpha}}{\mathrm{d}t^{\alpha}}p_{m,0}^{\alpha,\nu}(t)&=-(\lambda_{0}+\nu)p_{m,0}^{\alpha,\nu}(t)+\mu_{1}p_{m,1}^{\alpha,\nu}(t)+\nu,\\
	\frac{\mathrm{d}^{\alpha}}{\mathrm{d}t^{\alpha}}p_{m,n}^{\alpha,\nu}(t)&=-(\lambda_{n}+\mu_{n}+\nu)p_{m,n}^{\alpha,\nu}(t)+\lambda_{n-1}p_{m,n-1}^{\alpha,\nu}(t)+\mu_{n+1}p_{m,n+1}^{\alpha,\nu}(t),\,n=1,2,\dots
\end{align*}
with initial condition $p_{m,n}^{\alpha,\nu}(0)=\delta_{m,n}$.
Here, $\frac{\mathrm{d}^{\alpha}}{\mathrm{d}t^{\alpha}}$ is the Caputo fractional derivative of order $0<\alpha\leq1$ defined in Section \ref{SubSecFractionalDerivative}.
\end{theorem}
\begin{proof}
For $m,n\in I$ and $t\geq0$, we have
\begin{equation}\label{pmn}
	p_{m,n}^{\alpha,\nu}(t)=\mathrm{Pr}\{\mathcal{N}^{\nu}(Y_{\alpha}(t))=n|\mathcal{N}^{\nu}(0)=m\}
	=\int_{0}^{\infty}p_{m,n}^{\nu}(y)\mathrm{Pr}\{Y_{\alpha}(t)\in\mathrm{d}y\}.
\end{equation}
On taking the Laplace transform of \eqref{pmn} and using \eqref{ltLnu12}, we get
\begin{align}\label{pmnz}
	\tilde{p}_{m,n}^{\alpha,\nu}(z)&=z^{\alpha-1}\int_{0}^{\infty}p_{m,n}^{\nu}(y)e^{-yz^{\alpha}}\mathrm{d}y
	=z^{\alpha-1}\tilde{p}_{m,n}^{\nu}(z^{\alpha}),\, z>0.
\end{align}
Similarly, on taking the Laplace transform of \eqref{DE1}, we have
\begin{align}\label{ltDE1}
	\left.
	\begin{aligned}
		z\tilde{p}_{m,0}^{\nu}(z)-p_{m,0}^{\nu}(0)&=-(\lambda_{0}+\nu)\tilde{p}_{m,0}^{\nu}(z)+\mu_{1}\tilde{p}_{m,1}^{\nu}(z)+\frac{\nu}{z},\\
		z\tilde{p}_{m,n}^{\nu}(z)-p_{m,n}^{\nu}(0)&=-(\lambda_{n}+\mu_{n}+\nu)\tilde{p}_{m,n}^{\nu}(z)+\lambda_{n-1}\tilde{p}_{m,n-1}^{\nu}(z)+\mu_{n+1}\tilde{p}_{m,n+1}^{\nu}(z).
	\end{aligned}
	\right\}
\end{align}
By using \eqref{pmnz}, \eqref{ltDE1} and the fact that $p_{m,n}^{\alpha,\nu}(0)=p_{m,n}^{\nu}(0)$, we obtain
{\small\begin{align}\label{ltDE2}
	\left. 
	\begin{aligned}
		z^{\alpha}\tilde{p}_{m,0}^{\alpha,\nu}(z)-z^{\alpha-1}p_{m,0}^{\alpha,\nu}(0)&=-(\lambda_{0}+\nu)\tilde{p}_{m,0}^{\alpha,\nu}(z)+\mu_{1}\tilde{p}_{m,1}^{\alpha,\nu}(z)+\frac{\nu}{z},\\
		z^{\alpha}\tilde{p}_{m,n}^{\alpha,\nu}(z)-z^{\alpha-1}p_{m,n}^{\alpha,\nu}(0)&=-(\lambda_{n}+\mu_{n}+\nu)\tilde{p}_{m,n}^{\alpha,\nu}(z)+\lambda_{n-1}\tilde{p}_{m,n-1}^{\alpha,\nu}(z)+\mu_{n+1}\tilde{p}_{m,n+1}^{\alpha,\nu}(z).
	\end{aligned}
	\right\}
\end{align}}
On taking the inverse Laplace transform of \eqref{ltDE2} and using \eqref{caputolp}, we get the required result.
\end{proof}

\begin{remark}
For $\alpha=1$ and $\nu=0$, the system of differential equations for the state probabilities of time-changed BDPC in Theorem \ref{thm1} reduces to that of BDPC (see Di Crescenzo {\it et al.} (2008), Eq. (1)) and time-changed BDP (see Kataria and Vishwakarma (2025b), Eq. (1.2)), respectively. 
\end{remark}

\begin{proposition}\label{prop1}
The Laplace transform of the state probabilities of time-changed BDPC and time-changed BDP are related as follows:
\begin{equation}\label{pmnzz}
	\tilde{p}_{m,n}^{\alpha,\nu}(z)=(\nu+z^{\alpha})^{\frac{1-\alpha}{\alpha}}\Big(z^{\alpha-1}\tilde{p}_{m,n}^{\alpha,0}((\nu+z^{\alpha})^{1/\alpha})+\frac{\nu}{z}\tilde{p}_{0,n}^{\alpha,0}((\nu+z^{\alpha})^{1/\alpha})\Big),\,z>0.
\end{equation}
\end{proposition}
\begin{proof}
On using Eq. (2.2) of Pakes (1997) in \eqref{pmnz}, we obtain
\begin{align}\label{pmnzzz}
	\tilde{p}_{m,n}^{\alpha,\nu}(z)
	&=z^{\alpha-1}\int_{0}^{\infty}e^{-(\nu +z^{\alpha})x}p_{m,n}(x)\mathrm{d}x+\nu z^{\alpha-1}\int_{0}^{\infty}\int_{0}^{x}e^{-(\nu y+xz^{\alpha})}{p}_{0,n}(y)\mathrm{d}y\mathrm{d}x\nonumber\\
	&=z^{\alpha-1}\tilde{p}_{m,n}(\nu+z^{\alpha})+\nu z^{\alpha-1}\int_{0}^{\infty}{p}_{0,n}(y)e^{-\nu y}\int_{y}^{\infty}e^{-xz^{\alpha}}\mathrm{d}x\mathrm{d}y\nonumber\\
	&=z^{\alpha-1}\tilde{p}_{m,n}(\nu+z^{\alpha})+\frac{\nu}{z}\int_{0}^{\infty}e^{-(\nu+z^{\alpha})y}{p}_{0,n}(y)\mathrm{d}y\nonumber\\
	&=z^{\alpha-1}\tilde{p}_{m,n}(\nu+z^{\alpha})+\frac{\nu}{z}\tilde{p}_{0,n}(\nu+z^{\alpha}).
\end{align}
Finally, on using $\tilde{p}^{\alpha,0}_{m,n}(z)=z^{\alpha-1}\tilde{p}_{m,n}(z^{\alpha})$ in \eqref{pmnzzz}, we get the required result.
\end{proof}

\begin{remark}
On substituting $\alpha=1$ in \eqref{pmnzz}, the Laplace transform for state probabilities of time-changed BDPC reduces to that of BDPC (see Pakes (1997), Eq. (3.11)).
\end{remark}

The proof of the following result follows similar lines to that of Theorem 4.1 of Ascione {\it et al.} (2020). 
\begin{theorem}\label{thm3.2}
The inter-occurrence times of catastrophes in time-changed BDPC are Mittag-Leffler distributed with parameters $\alpha\in(0,1]$ and $\nu>0$.
\end{theorem}
\begin{proof}
In BDPC, it is known that the catastrophe occurs according to Poisson process, hence the first occurrence time of catastrophe is exponentially distributed with rate $\nu>0$ (see Pakes (1997) and Di Crescenzo {\it et al.} (2008)). To obtain the occurrence time of catastrophe in time-changed BDPC, let us define the modified process $\{X(t)\}_{t\geq0}$ such that it starts at $k\geq0$ and reaches its absorbing state $-1$ with rate $\nu$. This process is the modification of $\{\mathcal{ N}^{\nu}(t)\}_{t\geq0}$ such that birth and death rates are taken to be null. Note that catastrophe transitions BDPC $\{\mathcal{N}^{\nu}(t)\}_{t\geq0}$ from state $k$ to $0$ and it transition the process $\{X(t)\}_{t\geq0}$ from state $k$ to an absorbing state $-1$ with rate $\nu$. The state probabilities $x_{n}(t)=\mathrm{Pr}\{X(t)=n|X(0)=k\}$, $n\in\{-1,k\}$ solve the following system of differential equations:
\begin{align}\label{DE3}
	\left.
	\begin{aligned}
		\frac{\mathrm{d}}{\mathrm{d}t}x_{-1}(t)&=\nu x_{k}(t),\\
		\frac{\mathrm{d}}{\mathrm{d}t}x_{k}(t)&=-\nu x_{k}(t)
	\end{aligned}
	\right\}
\end{align}
with $x_{k}(0)=1$. So, $x_{k}(t)+x_{-1}(t)=1$. Now, let us denote $T$ as the occurrence time of the first catastrophe when $\mathcal{ N}^{\nu}(0)=k$. Then, its distribution function is given by 
\begin{align*}
	F_{T}(t)\coloneqq\mathrm{Pr}\{T\leq t\}&=1-\mathrm{Pr}\{T>t\}\nonumber\\
	&=1-\mathrm{Pr}\{X(t)=k|X(0)=k\}\nonumber\\
	&=1-x_{k}(t).
\end{align*}
On solving \eqref{DE3}, it can be shown that $x_{k}(t)=e^{-\nu t}$ and hence, $T$ has an exponential distribution with rate $\nu$.

To obtain the distribution of the occurrence time of catastrophe in $\{\mathcal{N}^{\alpha,\nu}(t)\}_{t\geq0}$,
we define a time-changed variant of $\{X(t)\}_{t\geq0}$ as follows: $X^{\alpha}(t)\coloneqq X(Y_{\alpha}(t))$, $0<\alpha<1$ and $X^{1}(t)=X(t)$. Here, $\{Y_{\alpha}(t)\}_{t\geq0}$ is the inverse stable subordinator such that $Y_{\alpha}(0)=0$, and it is independent of $\{X(t)\}_{t\geq0}$. Let us denote its state probabilities by $x^{\alpha}_{n}(t)=\mathrm{Pr}\{X^{\alpha}(t)=n|X^{\alpha}(0)=k\}$, $n\in\{-1,k\}$. Its Laplace transform is given by
\begin{align}\label{xmz}
	\tilde{x}_{n}^{\alpha}(z)
	&=\int_{0}^{\infty}e^{-zt}\mathrm{Pr}\{X(Y_{\alpha}(t))=n|X(0)=k\}\mathrm{d}t\nonumber\\
	&=\int_{0}^{\infty}\int_{0}^{\infty}e^{-zt}x_{n}(y)\mathrm{Pr}\{Y_{\alpha}(t)\in\mathrm{d}y\}\mathrm{d}t\nonumber\\
	&=z^{\alpha-1}\int_{0}^{\infty}x_{n}(y)e^{-yz^{\alpha}}\mathrm{d}y\nonumber\\
	&=z^{\alpha-1}\tilde{x}_{n}(z^{\alpha}),\, z>0,
\end{align}
where we have used \eqref{ltLnu12} in the penultimate step.

On taking the Laplace transform of \eqref{DE3}, and by using \eqref{caputolp} and \eqref{xmz}, we get
\begin{align*}
	\frac{\mathrm{d}^{\alpha}}{\mathrm{d}t^{\alpha}}x_{-1}^{\alpha}(t)&=\nu x_{k}^{\alpha}(t),\nonumber\\
	\frac{\mathrm{d}^{\alpha}}{\mathrm{d}t^{\alpha}}x_{k}^{\alpha}(t)&=-\nu x_{k}^{\alpha}(t)
\end{align*}
with initial condition $x_{k}^{\alpha}(0)=1$. Its solution is $x_{k}^{\alpha}(t)=E_{\alpha}(-\nu t^{\alpha})$ (see Kilbas {\it et al.} (2006)). Let $T_{\alpha}$ be the occurrence time of the first catastrophe in $\{\mathcal{ N}^{\alpha,\nu}(t)\}_{t\geq0}$ such that $\mathcal{ N}^{\alpha,\nu}(0)=k$. Then,
\begin{align*}
	F_{T_{\alpha}}(t)\coloneqq\mathrm{Pr}\{T_{\alpha}\leq t\}&=1-\mathrm{Pr}\{T_{\alpha}>t\}\nonumber\\
	&=1-\mathrm{Pr}\{X^{\alpha}(t)=k|X^{\alpha}(0)=k\}\nonumber\\
	&=1-E_{\alpha}(-\nu t^{\alpha}).
\end{align*}
As $\{\mathcal{ N}^{\nu}(t)\}_{t\geq0}$ is a Markov process, its time-changed variant $\{\mathcal{ N}^{\alpha,\nu}(t)\}_{t\geq0}$ is a semi-Markov process. Let $J_n$ be the $n$th jump time of $\{\mathcal{ N}^{\alpha,\nu}(t)\}_{t\geq0}$. Then, $\mathcal{ N}^{\alpha,\nu}(J_n)$ is a time-homogeneous Markov chain (see Gikhman and Skorokhod (1975)). Hence, the inter-occurrence times of catastrophes are Mittag-Leffler distributed with parameters $\alpha$ and $\nu$.
\end{proof}


\begin{remark}
For $\alpha=1$ in Theorem \ref{thm3.2}, the distribution of inter-occurrence times of catastrophes in time-changed BDPC reduces to that of BDPC (see Di Crescenzo {\it et al.} (2008)).
\end{remark}

The time spent by the process in any state before its transition to another state is called the sojourn time in that state.

\begin{theorem}\label{thmSojourn}
For any state $n>0$ of time-changed BDPC, the sojourn time in state $n$ is Mittag-Leffler distributed with parameters $\alpha$ and $\lambda_{n}+\mu_{n}+\nu$.
\end{theorem}
\begin{proof}
The proof follows similar lines to that of the Theorem \ref{thm3.2}. Thus, it is omitted.
\end{proof}

Let us define the first visit time of time-changed BDPC $\{\mathcal{N}^{\alpha,\nu}(t)\}_{t\geq0}$ to state zero as follows:
\begin{equation*}
	T^{\alpha,\nu}_{m,0}\coloneqq\inf\{t\geq0:\mathcal{N}^{\alpha,\nu}(t)=0\},
\end{equation*}
where $\mathcal{N}^{\alpha,\nu}(0)=m>0$. Also, let $T_{m,0}^{\alpha,0}$ be the first visit time of time-changed BDP $\{\mathcal{ N}^{\alpha,0}(t)\}_{t\geq0}$ to state zero.

\begin{theorem}\label{thm3}
Let $\mathcal{N}^{\alpha,\nu}(t)=\mathcal{ N}^{\nu}(Y_{\alpha}(t))$ be the time-changed BPPC such that zero is an absorbing state and $\mathcal{ N}^{\alpha,\nu}(0)=m>0$. Also, let $Z$ be the Mittag-Leffler random variable with parameters $\alpha$ and $\nu$ which is conditionally independent of $T_{m,0}^{\alpha,0}$ given $Y_{\alpha}(t)$ for all $t>0$. Then, 
\begin{equation}\label{TT}
	T_{m,0}^{\alpha,\nu}\overset{d}{=} \min\{T_{m,0}^{\alpha,0},Z\},
\end{equation}
where $\overset{d}{=}$ denotes equality in distribution.
\end{theorem}
\begin{proof}
Let $T_{m,0}^{\nu}$ and $T_{m,0}$ be the first visit time to state zero for BDPC $\{\mathcal{ N}^{\nu}(t)\}_{t\geq0}$ and BDP $\{\mathcal{ N}(t)\}_{t\geq0}$, respectively. Consider
\begin{align}\label{survTmoalphanu}
	\mathrm{Pr}\{T_{m,0}^{\alpha,\nu}>t\}
	&=\mathrm{Pr}\{\mathcal{ N}^{\alpha,\nu}(t)\neq0\,|\mathcal{ N}^{\alpha,\nu}(0)=m\}\nonumber\\
	&=1-\mathrm{Pr}\{\mathcal{ N}^{\nu}(Y_{\alpha}(t))=0\,|\mathcal{ N}^{\nu}(0)=m\}\nonumber\\
	&=1-\int_{0}^{\infty}\mathrm{Pr}\{\mathcal{ N}^{\nu}(y)=0|\mathcal{ N}^{\nu}(0)=m\}\mathrm{Pr}\{Y_{\alpha}(t)\in\mathrm{d}y\}\nonumber\\
	&=\int_{0}^{\infty}\mathrm{Pr}\{T_{m,0}^{\nu}>y\}\mathrm{Pr}\{Y_{\alpha}(t)\in\mathrm{d}y\}\nonumber\\
	&=\int_{0}^{\infty}e^{-\nu y}\mathrm{Pr}\{T_{m,0}>y\}\mathrm{Pr}\{Y_{\alpha}(t)\in\mathrm{d}y\},
\end{align}
where the last equality holds true because $T_{m,0}^{\nu}\overset{d}{=}\min\{T_{m,0},X\}$ such that $X$ is exponentially distributed with rate $\nu$ and independent of $T_{m,0}$ (see Di Crescenzo {\it et al.} (2008), p. 2250).

Let $\theta_{m,0}^{\alpha,\nu}=\min\{T_{m,0}^{\alpha,0},Z\}$.
Then,
\begin{align}\label{theta}
	\mathrm{Pr}\{\theta_{m,0}^{\alpha,\nu}>t\}
	&=\mathrm{Pr}\{T_{m,0}^{\alpha,0}>t,Z>t\}\nonumber\\
	&=\int_{0}^{\infty}\mathrm{Pr}\{T_{m,0}^{\alpha,0}>t,Z>t|Y_{\alpha}(t)\in\mathrm{d}y\}\mathrm{Pr}\{Y_{\alpha}(t)\in\mathrm{d}y\}\nonumber\\
	&=\int_{0}^{\infty}\mathrm{Pr}\{T_{m,0}^{\alpha,0}>t|Y_{\alpha}(t)\in\mathrm{d}y\}\mathrm{Pr}\{Z>t|Y_{\alpha}(t)\in\mathrm{d}y\}\mathrm{Pr}\{Y_{\alpha}(t)\in\mathrm{d}y\},
\end{align}
where the last step follows from the conditional independence of $T_{m,0}^{\alpha,0}$ and $Z$ given $Y_{\alpha}(t)$.

Now consider
\begin{align}\label{Tmoalpha}
	\mathrm{Pr}\{T_{m,0}^{\alpha,0}>t|Y_{\alpha}(t)\in\mathrm{d}y\}
	&=1-\mathrm{Pr}\{\mathcal{N}^{\alpha}(t)=0|\mathcal{ N}(0)=m,Y_{\alpha}(t)\in\mathrm{d}y\}\nonumber\\
	&=1-\mathrm{Pr}\{\mathcal{N}(y)=0|\mathcal{ N}(0)=m,Y_{\alpha}(t)\in\mathrm{d}y\}\nonumber\\
	&=1-\mathrm{Pr}\{\mathcal{N}(y)=0|\mathcal{ N}(0)=m\}\nonumber\\
	&=\mathrm{Pr}\{T_{m,0}>y\},
\end{align}
where the penultimate step follows from the independence of $\mathcal{N}(y)$ and $Y_{\alpha}(t)$. 

Let us assume that the exponential random variable $X$ be independent of the stable subordinator $D_{\alpha}(t)$, $0<\alpha<1$, $t\geq0$. On using Lemma 7.1 of Ascione {\it et al.} (2020), we get
\begin{align}\label{Zz}
	\mathrm{Pr}\{Z>t|Y_{\alpha}(t)\in\mathrm{d}y\}
	&=\mathrm{Pr}\{D_{\alpha}(X)>t|Y_{\alpha}(t)\in\mathrm{d}y\}\nonumber\\
	&=\mathrm{Pr}\{X>Y_{\alpha}(t)|Y_{\alpha}(t)\in\mathrm{d}y\}\nonumber\\
	&=\mathrm{Pr}\{X>y|Y_{\alpha}(t)\in\mathrm{d}y\}\nonumber\\
	&=\mathrm{Pr}\{X>y\}=e^{-\nu y}.
\end{align}
On substituting \eqref{Tmoalpha} and \eqref{Zz} in \eqref{theta}, we have
\begin{equation}\label{thetaf}
	\mathrm{Pr}\{\theta_{m,0}^{\alpha,\nu}>t\}=\int_{0}^{\infty}e^{-\nu y}\mathrm{Pr}\{T_{m,0}>y\}\mathrm{Pr}\{Y_{\alpha}(t)\in\mathrm{d}y\}.
\end{equation}
Finally, the required result follows from \eqref{survTmoalphanu} and \eqref{thetaf}.
\end{proof}

\begin{remark}
For $\alpha=1$ in \eqref{TT}, the distribution of the first visit time to state zero for time-changed BDPC reduces to that of BDPC (see Di Crescenzo {\it et al.} (2008), p. 2250).
\end{remark}

Let $f_{m,0}^{\alpha,\nu}(t)$, $f_{m,0}^{\alpha,0}(t)$ and $f_{m,0}(t)$ be the pdfs of $T_{m,0}^{\alpha,\nu}$, $T_{m,0}^{\alpha,0}$ and $T_{m,0}$, respectively. Also, let $\tilde{f}_{m,0}^{\alpha,\nu}(z)$, $\tilde{f}_{m,0}^{\alpha,0}(z)$ and $\tilde{f}_{m,0}(z)$ be their respective Laplace transforms.

\begin{corollary}\label{Corr1}
Let zero be an absorbing state of BDPC. Then,
\begin{equation}\label{fmoalphanuu}
	\tilde{f}_{m,0}^{\alpha,\nu}(z)=\frac{\nu}{\nu+z^{\alpha}}+\frac{z^{\alpha}}{\nu+z^{\alpha}}\tilde{f}_{m,0}(\nu+z^{\alpha}).
\end{equation}
\end{corollary}
\begin{proof}
For $m>0$, we have
\begin{align}\label{fmoalphanu}
	f_{m,0}^{\alpha,\nu}(t)
	&=\frac{\mathrm{d}}{\mathrm{d}t}\mathrm{Pr}\{T_{m,0}^{\alpha,\nu}\leq t\}\nonumber\\
	&=-\int_{0}^{\infty}e^{-\nu y}\mathrm{Pr}\{T_{m,0}>y\}\frac{\mathrm{d}}{\mathrm{d}t}\mathrm{Pr}\{Y_{\alpha}(t)\in\mathrm{d}y\},
\end{align}
where we have used \eqref{survTmoalphanu}.
On taking the Laplace transform of \eqref{fmoalphanu} and using \eqref{ltLnu12}, we get
\begin{align*}
	\tilde{f}_{m,0}^{\alpha,\nu}(z)
	&=-\int_{0}^{\infty}e^{-\nu y}\mathrm{Pr}\{T_{m,0}>y\}(z^{\alpha}e^{-z^{\alpha}y}\mathrm{d}y-\mathrm{Pr}\{Y_{\alpha}(0)\in\mathrm{d}y\})\nonumber\\
	&=1-z^{\alpha}\int_{0}^{\infty}\int_{y}^{\infty}e^{-(\nu+z^{\alpha})y}f_{m,0}(x)\mathrm{d}x\mathrm{d}y\nonumber\\
	&=1-z^{\alpha}\int_{0}^{\infty}f_{m,0}(x)\int_{0}^{x}e^{-(\nu+z^{\alpha})y}\mathrm{d}y\mathrm{d}x\nonumber\\
	&=1+\frac{z^{\alpha}}{\nu+z^{\alpha}}(\tilde{f}_{m,0}(\nu+z^{\alpha})-1).
\end{align*}
Thus, we have the required result.
\end{proof}

\begin{remark}
For $\alpha=1$ in \eqref{fmoalphanuu}, the Laplace transform of pdf of the first visit time to state zero for time-changed BDPC reduces to that of BDPC (see Di Crescenzo {\it et al.} (2008), Eq. (6)).
\end{remark}

\begin{remark}
Let zero be an absorbing state of $\{\mathcal{ N}^{\alpha,\nu}(t)\}_{t\geq0}$. Then,
\begin{align*}
	\mathbb{E}(T_{m,0}^{\alpha,\nu})
	&=\int_{0}^{\infty}\mathrm{Pr}\{T_{m,0}^{\alpha,\nu}>t\}\mathrm{d}t\nonumber\\
	&=\int_{0}^{\infty}\int_{0}^{\infty}\mathrm{Pr}\{T_{m,0}^{\nu}>y\}\mathrm{Pr}\{Y_{\alpha}(t)\in\mathrm{d}y\}\mathrm{d}t\ \ (\text{from \eqref{survTmoalphanu}})\nonumber\\
	&=\int_{0}^{\infty}\mathrm{Pr}\{T_{m,0}^{\nu}>Y_{\alpha}(t)\}\mathrm{d}t\nonumber\\
	&=\int_{0}^{\infty}\mathrm{Pr}\{D_{\alpha}(T_{m,0}^{\nu})>t\}\mathrm{d}t =\mathbb{E}(D_{\alpha}(T_{m,0}^{\nu}))=\infty,
\end{align*}
as the stable subordinator $D_{\alpha}(t)$, $t\geq0$ has infinite expectation.
\end{remark}

\subsection{Effective catastrophe of time-changed BDPC}\label{SecEffctCatastrophe}
It is important to note that there is a possibility of catastrophe in the time-changed BDPC at state zero. Throughout this section, we consider that the catastrophe occurs at positive states only. 
A catastrophe that occurs at a non-zero state is called an effective catastrophe. That is, we exclude the occurrence of catastrophe at state zero. Thus, the first occurrence time of an effective catastrophe is not Mittag-Leffler distributed.

Let $K_{m,0}^{\alpha,\nu}$ be the first occurrence time of an effective catastrophe when $\mathcal{N}^{\alpha,\nu}(0)=m$, $m\in I$. An effective catastrophe transitions the process from any positive state $n$ to state zero with rate $\nu$. As $T_{m,0}^{\alpha,\nu}\geq K_{m,0}^{\alpha,\nu}$, the expectation and variance of the first occurrence time of an effective catastrophe in time-changed BDPC is infinite.

Let us recall the modified birth-death process with catastrophe denoted by $\{\mathcal{M}^{\nu}(t)\}_{t\geq0}$ with state space $I\cup\{-1\}$ as discussed in Di Crescenzo {\it et al.} (2008). It is similar to $\{\mathcal{N}^{\nu}(t)\}_{t\geq0}$ except that catastrophe transitions the process in any positive state $n$ to absorbing state $-1$. The connection between the $K_{m,0}^{\alpha,\nu}$ and $\{\mathcal{M}^{\nu}(t)\}_{t\geq0}$ is due its transition from any positive state $n$ to $-1$ is same as transition of $\{\mathcal{N}^{\nu}(t)\}_{t\geq0}$ from $n>0$ to state zero. The state probabilities $q_{m,n}^{\nu}(t)=\mathrm{Pr}\{\mathcal{M}^{\nu}(t)=n|\mathcal{M}^{\nu}(0)=m\}$, $m\in I$, $n\in$ $I\cup\{-1\}$ of $\{\mathcal{M}^{\nu}(t)\}_{t\geq0}$ solve the following system of differential equations:
{\small\begin{align}\label{DE5}
	\left.
	\begin{aligned}
		\frac{\mathrm{d}}{\mathrm{d}t}q_{m,-1}^{\nu}(t)&=\nu(1-q_{m,-1}^{\nu}(t)-q_{m,0}^{\nu}(t)),\\
		\frac{\mathrm{d}}{\mathrm{d}t}q_{m,0}^{\nu}(t)&=-\lambda_{0}q_{m,0}^{\nu}(t)+\mu_{1}q_{m,1}^{\nu}(t),\\
		\frac{\mathrm{d}}{\mathrm{d}t}q_{m,n}^{\nu}(t)&=-(\lambda_{n}+\mu_{n}+\nu)q_{m,n}^{\nu}(t)+\lambda_{n-1}q_{m,n-1}^{\nu}(t)+\mu_{n+1}q_{m,n+1}^{\nu}(t),\,n=1,2,\dots
	\end{aligned}
	\right\}
\end{align}}
with initial condition $q_{m,n}^{\nu}(0)=\delta_{m,n}$.
Consider a time-changed variant of it as follows: $\mathcal{M}^{\alpha,\nu}(t)=\mathcal{M}^{\nu}(Y_{\alpha}(t))$, where the inverse stable subordinator $\{Y_{\alpha}(t)\}_{t\geq0}$ is independent of $\{\mathcal{M}^{\nu}(t)\}_{t\geq0}$. 
Let $q_{m,n}^{\alpha,\nu}(t)=\mathrm{Pr}\{\mathcal{M}^{\alpha,\nu}(t)=n|\mathcal{M}^{\alpha,\nu}(0)=m\}$, $m\in I$, $n\in I\cup\{-1\}$ be its state probabilities.
The catastrophic transition for $\{\mathcal{M}^{\alpha,\nu}(t)\}_{t\geq0}$ from $n>0$ to $-1$ corresponds to the transition of $\{\mathcal{N}^{\alpha,\nu}(t)\}_{t\geq0}$ from $n>0$ to $0$. Let $g_{m,0}^{\alpha,\nu}(t)$ be the pdf of $K_{m,0}^{\alpha,\nu}$. Then, we have
\begin{equation}\label{Km0}
	\mathrm{Pr}\{K_{m,0}^{\alpha,\nu}>t\}=\int_{t}^{\infty}g_{m,0}^{\alpha,\nu}(x)\mathrm{d}x=\sum_{n=0}^{\infty}q_{m,n}^{\alpha,\nu}(t)=1-q_{m,-1}^{\alpha,\nu}(t),\,m\in I.
\end{equation}
On taking the Laplace transform on both sides of
\begin{equation*}
	q_{m,n}^{\alpha,\nu}(t)=\int_{0}^{\infty}q_{m,n}^{\nu}(y)\mathrm{Pr}\{Y_{\alpha}(t)\in\mathrm{d}y\}
\end{equation*}
and by using \eqref{ltLnu12}, we get
\begin{equation}\label{qmn}
	\tilde{q}_{m,n}^{\alpha,\nu}(z)=z^{\alpha-1}\tilde{q}_{m,n}^{\nu}(z^{\alpha}).
\end{equation}
On taking the Laplace transform of \eqref{DE5}, and by using \eqref{caputolp} and \eqref{qmn}, we obtain the following system of fractional differential equations:
{\small\begin{align}\label{DE6}
	\left.
	\begin{aligned}
		\frac{\mathrm{d}^{\alpha}}{\mathrm{d}t^{\alpha}}q_{m,-1}^{\alpha,\nu}(t)&=\nu(1-q_{m,-1}^{\alpha,\nu}(t)-q_{m,0}^{\alpha,\nu}(t)),\\
		\frac{\mathrm{d}^{\alpha}}{\mathrm{d}t^{\alpha}}q_{m,0}^{\alpha,\nu}(t)&=-\lambda_{0}q_{m,0}^{\alpha,\nu}(t)+\mu_{1}q_{m,1}^{\alpha,\nu}(t),\\
		\frac{\mathrm{d}^{\alpha}}{\mathrm{d}t^{\alpha}}q_{m,n}^{\alpha,\nu}(t)&=-(\lambda_{n}+\mu_{n}+\nu)q_{m,n}^{\alpha,\nu}(t)+\lambda_{n-1}q_{m,n-1}^{\alpha,\nu}(t)+\mu_{n+1}q_{m,n+1}^{\alpha,\nu}(t),\,n=1,2,\dots
	\end{aligned}
	\right\}
\end{align}}
with initial condition $q_{m,n}^{\alpha,\nu}(0)=\delta_{m,n}$.

\begin{theorem}\label{thm2}
For $m\in I$ and $z>0$, we have
\begin{equation}\label{qm1z}
	\tilde{q}^{\alpha,\nu}_{m,-1}(z)=\frac{\nu z^{\alpha-1}}{(\nu+z^{\alpha})}\Bigg(\frac{1}{z^{\alpha}}-\frac{(\nu+z^{\alpha})^{\frac{(1-\alpha)}{\alpha}}\tilde{p}_{m,0}^{\alpha,0}((\nu+z^{\alpha})^{1/\alpha})}{1-\nu(\nu+z^{\alpha})^{\frac{(1-\alpha)}{\alpha}}\tilde{p}_{0,0}^{\alpha,0}((\nu+z^{\alpha})^{1/\alpha})}\Bigg)
\end{equation}
and
\begin{align*}
	\tilde{q}^{\alpha,\nu}_{m,n}(z)&=z^{\alpha-1}(\nu+z^{\alpha})^{\frac{(1-\alpha)}{\alpha}}\Bigg(\tilde{p}_{m,n}^{\alpha,0}((\nu+z^{\alpha})^{1/\alpha})\nonumber\\
	&\hspace{4cm}
	+\frac{\nu(\nu+z^{\alpha})^{\frac{(1-\alpha)}{\alpha}}\tilde{p}_{m,0}^{\alpha,0}((\nu+z^{\alpha})^{1/\alpha})\tilde{p}_{0,n}^{\alpha,0}((\nu+z^{\alpha})^{1/\alpha})}{1-\nu(\nu+z^{\alpha})^{\frac{(1-\alpha)}{\alpha}}\tilde{p}_{0,0}^{\alpha,0}((\nu+z^{\alpha})^{1/\alpha})} \Bigg),\ \ n\in I.
\end{align*}
\end{theorem}
\begin{proof}
On using Eq. (11) and Eq. (12) of Di Crescenzo {\it et al.} (2008) in \eqref{qmn}, we obtain
\begin{equation}\label{qm1zz}
\tilde{q}_{m,-1}^{\alpha,\nu}(z)=\frac{\nu z^{\alpha-1}}{(\nu+z^{\alpha})}\Big(\frac{1}{z^{\alpha}}-\frac{\tilde{p}_{m,0}(\nu+z^{\alpha})}{1-\nu\tilde{p}_{0,0}(\nu+z^{\alpha})}\Big)
\end{equation}
and
\begin{equation}\label{qmnzz}
\tilde{q}_{m,n}^{\alpha,\nu}(z)=z^{\alpha-1}\Bigg(\tilde{p}_{m,n}(\nu+z^{\alpha})+\frac{\nu \tilde{p}_{0,n}(\nu+z^{\alpha})\tilde{p}_{m,0}(\nu+z^{\alpha})}{1-\nu\tilde{p}_{0,0}(\nu+z^{\alpha})}\Bigg),
\end{equation}
respectively.
By using  $\tilde{p}^{\alpha,0}_{m,n}(z)=z^{\alpha-1}\tilde{p}_{m,n}(z^{\alpha})$ in \eqref{qm1zz} and \eqref{qmnzz}, we get the required result.
\end{proof}

\begin{remark}
On substituting $\alpha=1$ in Theorem \ref{thm2}, the Laplace transform of the state probabilities of $\{\mathcal{M}^{\alpha,\nu}(t)\}_{t\geq0}$ reduces to that of $\{\mathcal{M}^{\nu}(t)\}_{t\geq0}$ (see Di Crescenzo {\it et al.} (2008), Eq. (11) and Eq. (12)).
\end{remark}

\begin{theorem}\label{thm3.6}
The Laplace transform of the pdf of first occurrence time of an effective catastrophe in time-changed BPDC is given by
{\small\begin{equation*}
	\tilde{g}_{m,0}^{\alpha,\nu}(z)=\frac{\nu z^{\alpha}}{\nu+z^{\alpha}}\Big(\frac{1}{z^{\alpha}}-\frac{(\nu+z^{\alpha})^{\frac{(1-\alpha)}{\alpha}}\tilde{f}_{m,0}({\nu+z^{\alpha}})\tilde{p}_{0,0}^{\alpha,0}((\nu+z^{\alpha})^{1/\alpha})}{1-\nu(\nu+z^{\alpha})^{\frac{(1-\alpha)}{\alpha}} \tilde{p}_{0,0}^{\alpha,0}((\nu+z^{\alpha})^{1/\alpha})}\Big),\,z>0,
\end{equation*}}
where $\mathcal{ N}^{\alpha,\nu}(0)=m$.
\end{theorem}
\begin{proof}
From \eqref{Km0}, we have
\begin{equation}\label{gm0}
	g_{m,0}^{\alpha,\nu}(t)=\frac{\mathrm{d}}{\mathrm{d}t}q_{m,-1}^{\alpha,\nu}(t).
\end{equation}
On taking the Laplace transform of \eqref{gm0}, we get
\begin{equation}\label{gm0z}
	\tilde{g}_{m,0}^{\alpha,\nu}(z)=z\tilde{q}_{m,-1}^{\alpha,\nu}(z)-q_{m,-1}^{\alpha,\nu}(0).
\end{equation} 
By using \eqref{qm1z} in \eqref{gm0z}, we obtain
\begin{equation}\label{gm0alpha}
	\tilde{g}_{m,0}^{\alpha,\nu}(z)=\frac{\nu z^{\alpha}}{\nu+z^{\alpha}}\Bigg(\frac{1}{z^{\alpha}}-\frac{(\nu+z^{\alpha})^{\frac{(1-\alpha)}{\alpha}}\tilde{p}_{m,0}^{\alpha,0}((\nu+z^{\alpha})^{1/\alpha})}{1-\nu(\nu+z^{\alpha})^{\frac{(1-\alpha)}{\alpha}}\tilde{p}_{0,0}^{\alpha,0}((\nu+z^{\alpha})^{1/\alpha})}\Bigg).
\end{equation}
Note that
\begin{equation}\label{bfpm0alphaa}
	\tilde{p}_{m,0}^{\alpha,0}(z)=z^{\alpha-1}\tilde{p}_{m,0}(z^{\alpha})
	=z^{\alpha-1}\tilde{f}_{m,0}(z^{\alpha})\tilde{p}_{0,0}(z^{\alpha})
	=\tilde{f}_{m,0}(z^{\alpha})\tilde{p}_{0,0}^{\alpha,0}(z^{\alpha}),
\end{equation}
where we have used Eq. (24) from Di Crescenzo {\it et al.} (2008) in the penultimate step.
On substituting \eqref{bfpm0alphaa} in \eqref{gm0alpha}, we get the required result.
\end{proof}

The following result follows on using Corollary \ref{Corr1} in Theorem \ref{thm3.6}:
\begin{corollary}
Let zero be an absorbing state in time-changed BDPC $\{\mathcal{ N}^{\alpha,\nu}(t)\}_{t\geq0}$. Then,
\begin{equation*}
	\tilde{g}_{m,0}^{\alpha,\nu}(z)=\tilde{f}_{m,0}^{\alpha,\nu}(z)-\frac{z^{\alpha}\tilde{f}_{m,0}(\nu+z^{\alpha})}{(\nu+z^{\alpha})\big(1-\nu(\nu+z^{\alpha})^{\frac{1-\alpha}{\alpha}}\tilde{p}_{0,0}^{\alpha,0}((\nu+z^{\alpha})^{1/\alpha})\big)}.
\end{equation*}
\end{corollary}

\begin{remark}
On substituting $\alpha=1$ in Theorem \ref{thm3.6}, the Laplace transform of the pdf of first occurrence time of an effective catastrophe in $\{\mathcal{N}^{\alpha,\nu}(t)\}_{t\geq0}$ reduces to that of $\{\mathcal{N}^{\nu}(t)\}_{t\geq0}$ (see Di Crescenzo {\it et al.} (2008), Proposition 3.1).
\end{remark}


\section{Time-Changed Linear Birth-Death Process with Catastrophe}\label{sec4}
In this section, we study a particular case of the time-changed BDPC which is obtained by taking the birth rates as $\lambda_{n}=n\lambda$ and death rates as $\mu_{n}=n\mu$ for $n\in\{0,1,2,\dots\}$. We call it the time-changed linear birth-death process with catastrophe, in short the time-changed LBDPC. It is important to note that zero becomes the absorbing state as $\lambda_0=0$. 

Let us denote the time-changed LBDPC by $\{N^{\alpha,\nu}(t)\}_{t\geq0}$ and its state probabilities by $h^{\alpha,\nu}_{m,n}(t)=\mathrm{Pr}\{N^{\alpha,\nu}(t)=n|N^{\alpha,\nu}(0)=m\}$. Thus, from Theorem \ref{thm1}, we have the following system of fractional differential equations:
{\footnotesize\begin{align}\label{DE7}
	\left.
	\begin{aligned}
		\frac{\mathrm{d}^{\alpha}}{\mathrm{d}t^{\alpha}}h^{\alpha,\nu}_{m,0}(t)
		&=-\nu h^{\alpha,\nu}_{m,0}(t)+\mu h^{\alpha,\nu}_{m,1}(t)+\nu,\\
		\frac{\mathrm{d}^{\alpha}}{\mathrm{d}t^{\alpha}}h^{\alpha,\nu}_{m,n}(t)
		&=-(n\lambda+n\mu+\nu)h^{\alpha,\nu}_{m,n}(t)+(n-1)\lambda h^{\alpha,\nu}_{m,n-1}(t) +(n+1)\mu h^{\alpha,\nu}_{m,n+1}(t),\,n=1,2,\dots
	\end{aligned}
	\right\}
\end{align}}
with initial condition $h^{\alpha,\nu}_{m,n}(0)=\delta_{m,n}$.

Let $\{{N}^{\alpha,0}(t)\}_{t\geq0}$ be the corresponding time-changed linear birth-death process without catastrophe, that is, the time-changed LBDP. Its state probabilities are denoted by ${h}^{\alpha,0}_{m,n}(t)=\mathrm{Pr}\{{N}^{\alpha,0}(t)=n|{N}^{\alpha,0}(0)=m\}$, $n\in I$.

For the subsequent results, we recall that $E_{\alpha}(\cdot)$, $E_{\alpha,\beta}(\cdot)$ and $E_{\alpha,\beta}^{\gamma}(\cdot)$ are the Mittag-Leffler functions as defined in Section \ref{SubSecMittagLeffler}. First, we obtain the extinction probabilities of time-changed LBDPC for three distinct cases.

\begin{theorem}\label{thmzeroprobforeqrates}
	For $\lambda=\mu$, the extinction probability of time-changed LBDPC $\{N^{\alpha,\nu}(t)\}_{t\geq0}$ is given by 
	\begin{equation*}
		h^{\alpha,\nu}_{1,0}(t)=E_{\alpha}(-\nu t^{\alpha})-\int_{0}^{\infty}e^{-x}E_{\alpha}(-(\nu+\lambda x)t^{\alpha})\mathrm{d}x+\nu t^{\alpha}E_{\alpha,\alpha+1}(-\nu t^\alpha).
	\end{equation*}
\end{theorem}
\begin{proof}
From Eq. (2.28) of Orsingher and Polito (2011), we have
\begin{equation}\label{p10alpha0}
	\tilde{h}^{\alpha,0}_{1,0}(z)=\lambda\int_{0}^{\infty}e^{-x}\frac{z^{\alpha-1}}{(z^{\alpha}+\lambda x)^{2}}\mathrm{d}x.
\end{equation}
As zero is the absorbing state in time-changed LBDP, we get
\begin{equation}\label{p1alpha0z}
	\tilde{h}^{\alpha,0}_{0,0}(z)=1/z.
\end{equation}
By using \eqref{p10alpha0} and \eqref{p1alpha0z} in Proposition \ref{prop1}, we obtain
\begin{equation}\label{p10alphanuz}
	\tilde{h}^{\alpha,\nu}_{1,0}(z)=\lambda\int_{0}^{\infty}e^{-x}\frac{z^{\alpha-1}}{(z^{\alpha}+\nu+\lambda x)^2}\mathrm{d}x+\nu\frac{z^{-1}}{z^{\alpha}+\nu}.
\end{equation}
On taking the inverse Laplace transform of \eqref{p10alphanuz}, we get
\begin{equation}\label{p10alphanut}
	h^{\alpha,\nu}_{1,0}(t)
	=\lambda\int_{0}^{\infty}e^{-x}\int_{0}^{t}y^{\alpha-1}E_{\alpha,\alpha}(-(\nu+\lambda x)y^{\alpha})E_{\alpha}(-(\nu+\lambda x)(t-y)^\alpha)\mathrm{d}y\mathrm{d}x+\nu t^{\alpha}E_{\alpha,\alpha+1}(-\nu t^{\alpha}).
\end{equation}
From Eq. (2.30) of Orsingher and Polito (2011), we have
\begin{equation}\label{ConvInt}
	\int_{0}^{t}y^{\alpha-1}E_{\alpha,\alpha}(-(\nu+\lambda x)y^{\alpha})E_{\alpha}(-(\nu+\lambda x)(t-y)^\alpha)\mathrm{d}y=\frac{t^{\alpha}}{\alpha}E_{\alpha,\alpha}(-(\nu+\lambda x)t^{\alpha}).
\end{equation}
On substituting \eqref{ConvInt} in \eqref{p10alphanut}, we obtain
\begin{align}\label{p10alphanutt}
	h^{\alpha,\nu}_{1,0}(t)
	&=\frac{\lambda t^{\alpha}}{\alpha}\int_{0}^{\infty}e^{-x}E_{\alpha,\alpha}(-(\nu+\lambda x)t^{\alpha})\mathrm{d}x+\nu t^{\alpha}E_{\alpha,\alpha+1}(-\nu t^{\alpha})\nonumber\\
	&=\frac{1}{\alpha}\int_{-\infty}^{-\nu t^{\alpha}}e^{\frac{ut^{-\alpha}+\nu}{\lambda}}E_{\alpha,\alpha}(u)\mathrm{d}u+\nu t^{\alpha}E_{\alpha,\alpha+1}(-\nu t^{\alpha})\nonumber\\
	&=\int_{-\infty}^{-\nu t^{\alpha}}e^{\frac{ut^{-\alpha}+\nu}{\lambda}}\frac{\mathrm{d}}{\mathrm{d}u}E_{\alpha}(u)\mathrm{d}u+\nu t^{\alpha}E_{\alpha,\alpha+1}(-\nu t^{\alpha})\nonumber\\
	&=E_{\alpha}(-\nu t^{\alpha})-\frac{1}{\lambda t^{\alpha}}\int_{-\infty}^{-\nu t^{\alpha}}e^{\frac{ut^{-\alpha}+\nu}{\lambda}}E_{\alpha}(u)\mathrm{d}u+\nu t^{\alpha}E_{\alpha,\alpha+1}(-\nu t^{\alpha}).
\end{align}
By change of variables in \eqref{p10alphanutt}, we get the required result.
\end{proof}

\begin{theorem}\label{thm4.1}
For $\lambda>\mu$, the extinction probability of time-changed LBDPC $\{N^{\alpha,\nu}(t)\}_{t\geq0}$ is given by 
\begin{equation*}
	h^{\alpha,\nu}_{1,0}(t)=\frac{\mu}{\lambda}E_{\alpha}(-\nu t^{\alpha})-\Big(\frac{\lambda-\mu}{\lambda}\Big)\sum_{k=1}^{\infty}\Big(\frac{\mu}{\lambda}\Big)^{k}E_{\alpha}(-(\nu+(\lambda-\mu)k) t^{\alpha})+\nu t^{\alpha}E_{\alpha,\alpha+1}(-\nu t^{\alpha}).
\end{equation*}
\end{theorem}
\begin{proof}
From Eq. (2.16) of Orsingher and Polito (2011), we have
\begin{equation}\label{bfp10z}
	\tilde{h}^{\alpha,0}_{1,0}(z)=\frac{\mu}{\lambda z}-\Big(\frac{\lambda-\mu}{\lambda}\Big)\sum_{k=1}^{\infty}\Big(\frac{\mu}{\lambda}\Big)^{k}\frac{z^{\alpha-1}}{z^{\alpha}+(\lambda-\mu)k}.
\end{equation}
By substituting \eqref{p1alpha0z} and \eqref{bfp10z} in Proposition \ref{prop1}, we obtain 
\begin{equation}\label{p10z}
	\tilde{h}^{\alpha,\nu}_{1,0}(z)= \frac{\mu}{\lambda}\frac{z^{\alpha-1}}{z^{\alpha}+\nu}-\Big(\frac{\lambda-\mu}{\lambda}\Big)\sum_{k=1}^{\infty}\Big(\frac{\mu}{\lambda}\Big)^{k}\frac{z^{\alpha-1}}{z^{\alpha}+\nu+(\lambda-\mu)k}+\frac{\nu z^{-1}}{z^{\alpha}+\nu}.
\end{equation}
On taking the inverse Laplace transform of \eqref{p10z} and using \eqref{ltm12}, we get the required result.
\end{proof}

The proof of the next result follows similar lines to that of Theorem \ref{thm4.1}. Here, we use the following result (see Orsingher and Polito (2011), Eq. (2.23)):
\begin{equation*}\label{bfp10zz}
	\tilde{h}^{\alpha,0}_{1,0}(z)=\frac{1}{z}-\Big(\frac{\mu-\lambda}{\lambda}\Big)\sum_{k=1}^{\infty}\Big(\frac{\lambda}{\mu}\Big)^{k}\frac{z^{\alpha-1}}{z^{\alpha}+(\mu-\lambda)k}.
\end{equation*}

\begin{theorem}\label{thm4.2}
For $\lambda<\mu$, the extinction probability of time-changed LBDPC $\{N^{\alpha,\nu}(t)\}_{t\geq0}$ is given by 
\begin{equation*}
	h^{\alpha,\nu}_{1,0}(t)=E_{\alpha}(-\nu t^{\alpha})-\Big(\frac{\mu-\lambda}{\lambda}\Big)\sum_{k=1}^{\infty}\Big(\frac{\lambda}{\mu}\Big)^{k}E_{\alpha}(-(\nu+(\mu-\lambda)k)t^{\alpha})+\nu t^{\alpha}E_{\alpha,\alpha+1}(-\nu t^{\alpha}).
\end{equation*}
\end{theorem}

\begin{remark}
	On substituting $\nu=0$ in Theorem \ref{thmzeroprobforeqrates}, Theorem \ref{thm4.1} and Theorem \ref{thm4.2}, the extinction probability of time-changed LBPDC reduces to that of time-changed LBDP (see Orsingher and Polito (2011), Eq. (1.11)).
\end{remark}

Next, we obtain the state probabilities of time-changed LBDPC.

\begin{theorem}\label{thmnonzeroforeqrates}
For $\lambda=\mu$, the state probabilities of time-changed LBDPC $\{N^{\alpha,\nu}(t)\}_{t\geq0}$ are given by
\begin{equation*}
	h^{\alpha,\nu}_{1,n}(t)=\frac{(-\lambda)^{n-1}}{n!}\frac{\mathrm{d}^{n}}{\mathrm{d}\lambda^{n}}\big(\lambda\big(E_{\alpha}(-\nu t^{\alpha})+ \nu t^{\alpha}E_{\alpha,\alpha+1}(-\nu t^{\alpha})- h^{\alpha,\nu}_{1,0}(t)\big)\big),\, n\geq1.
\end{equation*}
\end{theorem}
\begin{proof}
On taking the Laplace transform of Eq. (3.20) of Orsingher and Polito (2011), we get
\begin{equation}\label{p1nalpha}
	\tilde{h}^{\alpha,0}_{1,n}(z)=\frac{(-\lambda)^{n-1}}{n!}\frac{\mathrm{d}^{n}}{\mathrm{d}\lambda^{n}}\Big(\lambda\Big(\frac{1}{z}-\tilde{h}^{\alpha,0}_{1,0}(z)\Big)\Big).
\end{equation}
In the time-changed LBDP, zero is the absorbing state. Thus, we have
\begin{equation}\label{p1alpha0nz}
	\tilde{h}^{\alpha,0}_{0,n}(z)=0,\,n>0.
\end{equation}
On substituting \eqref{p1nalpha} and \eqref{p1alpha0nz} in Proposition \ref{prop1}, we obtain
\begin{align}\label{p1nalphaz}
	\tilde{h}^{\alpha,\nu}_{1,n}(z)
	&=\frac{(-\lambda)^{n-1}z^{\alpha-1}(z^{\alpha}+\nu)^{\frac{1-\alpha}{\alpha}}}{n!}\frac{\mathrm{d}^{n}}{\mathrm{d}\lambda^{n}}\Big(\lambda\Big(\frac{1}{(z^{\alpha}+\nu)^{1/\alpha}}-\tilde{h}^{\alpha,0}_{1,0}((z^{\alpha}+\nu)^{1/\alpha})\Big)\Big)\nonumber\\
	&=\frac{(-\lambda)^{n-1}}{n!}\frac{\mathrm{d}^{n}}{\mathrm{d}\lambda^{n}}\Big(\lambda\Big(\frac{z^{\alpha-1}}{z^{\alpha}+\nu}-\lambda\int_{0}^{\infty}e^{-x}\frac{z^{\alpha-1}}{(z^{\alpha}+\nu+\lambda x)^2}\mathrm{d}x\Big)\Big),
\end{align}
where we have used \eqref{p10alpha0} in the last step.
On taking the inverse Laplace transform of \eqref{p1nalphaz}, we get
\begin{align}\label{p1nalphat}
	h^{\alpha,\nu}_{1,n}(t)
	&=\frac{(-\lambda)^{n-1}}{n!}\frac{\mathrm{d}^{n}}{\mathrm{d}\lambda^{n}}\Big(\lambda\Big(E_{\alpha}(-\nu t^{\alpha})-\lambda\int_{0}^{\infty}e^{-x}\int_{0}^{t}y^{\alpha-1}E_{\alpha,\alpha}(-(\nu+\lambda x)y^{\alpha})\nonumber\\
	&\hspace{7.5cm} \cdot E_{\alpha}(-(\nu+\lambda x)(t-y)^{\alpha})\mathrm{d}y\mathrm{d}x\Big)\Big).
\end{align}
By using \eqref{p10alphanut} in \eqref{p1nalphat}, the required result is obtained.
\end{proof}

\begin{theorem}\label{thm4.3}
For $\lambda>\mu$, the state probabilities of time-changed LBDPC $\{N^{\alpha,\nu}(t)\}_{t\geq0}$ are given by
\begin{equation*}
	h^{\alpha,\nu}_{1,n}(t)
	=\Big(\frac{\lambda-\mu}{\lambda}\Big)^{2}\sum_{k=0}^{\infty}\sum_{r=0}^{n-1}(-1)^{r}\binom{k+n}{k}\binom{n-1}{r}\Big(\frac{\mu}{\lambda}\Big)^{k}\nonumber
	E_{\alpha}(-(\nu+(\lambda-\mu)(k+r+1))t^{\alpha}),\, n\geq1.
\end{equation*}
\end{theorem}
\begin{proof}
From Eq. (3.3) of Orsingher and Polito (2011), we have
\begin{equation}\label{bfp1nz}
	\tilde{h}^{\alpha,0}_{1,n}(z)=\Big(\frac{\lambda-\mu}{\lambda}\Big)^{2}\sum_{k=0}^{\infty}\sum_{r=0}^{n-1}(-1)^{r}\binom{k+n}{k}\binom{n-1}{r}\Big(\frac{\mu}{\lambda}\Big)^{k}\frac{z^{\alpha-1}}{z^{\alpha}+(\lambda-\mu)(k+r+1)}.
\end{equation}
On substituting \eqref{p1alpha0nz} and \eqref{bfp1nz} in Proposition \ref{prop1}, we obtain
{\small\begin{equation}\label{p1nz}
	\tilde{h}^{\alpha,\nu}_{1,n}(z)
	=\Big(\frac{\lambda-\mu}{\lambda}\Big)^{2}\sum_{k=0}^{\infty}\sum_{r=0}^{n-1}(-1)^{r}\binom{k+n}{k}\binom{n-1}{r}\Big(\frac{\mu}{\lambda}\Big)^{k}\frac{z^{\alpha-1}}{z^{\alpha}+\nu+(\lambda-\mu)(k+r+1)}.
\end{equation}}
By taking the inverse Laplace transform of \eqref{p1nz} and using \eqref{ltm12}, we get the required result.
\end{proof}

The proof of the next result follows similar lines to that of Theorem \ref{thm4.3}. Here, we use the following result (see Orsingher and Polito (2011), Eq. (3.18)):
\begin{equation*}\label{bfp1nz2}
	\tilde{h}^{\alpha,0}_{1,n}(z)=\Big(\frac{\lambda}{\mu}\Big)^{n-1}\Big(\frac{\mu-\lambda}{\lambda}\Big)^{2}\sum_{k=0}^{\infty}\sum_{r=0}^{n-1}(-1)^{r}\binom{k+n}{k}\binom{n-1}{r}\Big(\frac{\lambda}{\mu}\Big)^{k}\frac{z^{\alpha-1}}{z^{\alpha}+(\mu-\lambda)(k+r+1)}.
\end{equation*}
\begin{theorem}\label{thm4.4}
For $\lambda<\mu$, the state probabilities of time-changed LBDPC $\{N^{\alpha,\nu}(t)\}_{t\geq0}$ are given by
{\footnotesize\begin{align*}
	h^{\alpha,\nu}_{1,n}(t)
	&=\Big(\frac{\lambda}{\mu}\Big)^{n-1}\Big(\frac{\mu-\lambda}{\mu}\Big)^{2}\sum_{k=0}^{\infty}\sum_{r=0}^{n-1}(-1)^{r}\binom{k+n}{k}\binom{n-1}{r}\Big(\frac{\lambda}{\mu}\Big)^{k} E_{\alpha}(-(\nu+(\mu-\lambda)(k+r+1))t^{\alpha}),\,n\geq1.
\end{align*}}
\end{theorem}
%

\begin{remark}
	For $\nu=0$ in Theorem \ref{thmnonzeroforeqrates}, Theorem \ref{thm4.3} and Theorem \ref{thm4.4}, the state probabilities of time-changed LBPDC reduces to that of time-changed LBDP (see Orsingher and Polito (2011), Eq. (1.13)).
\end{remark}

Let us define the probability generating function for time-changed LBDPC as follows:
\begin{equation}\label{pgf}
	G_{m}^{\alpha,\nu}(x,t)=\sum_{n=0}^{\infty}x^{n}h^{\alpha,\nu}_{m,n}(t),\,t\geq0,\,|x|\leq1.
\end{equation}
On taking the Caputo derivative of order $0<\alpha\leq1$ on both sides of \eqref{pgf} and using \eqref{DE7}, we obtain
\begin{equation}\label{pgfDE}
	\frac{\partial^{\alpha}}{\partial t^{\alpha}}G_{m}^{\alpha,\nu}(x,t)=\frac{\partial}{\partial x}G_{m}^{\alpha,\nu}(x,t)(\lambda x-\mu)(x-1)+\nu(1-G_{m}^{\alpha,\nu}(x,t)),
\end{equation}
with initial condition $G_{m}^{\alpha,\nu}(x,0)=x^{m}$.

\begin{theorem}
For $h^{\alpha,\nu}_{1,1}(0)=1$, the expectation of time-changed LBDPC is given by
\begin{equation}\label{expt}
	\mathbb{E}(N^{\alpha,\nu}(t))=E_{\alpha}((\lambda-\mu-\nu)t^{\alpha}).
\end{equation}
\end{theorem}
\begin{proof}
Note that
\begin{equation}\label{exppgf}
	\mathbb{E}(N^{\alpha,\nu}(t))=\frac{\partial}{\partial x}G_{1}^{\alpha,\nu}(x,t)\bigg|_{x=1}.
\end{equation}
On taking the first order derivative of \eqref{pgfDE} with respect to variable $x$ and using \eqref{exppgf}, we obtain
\begin{equation}\label{expDE}
	\frac{\mathrm{d}^{\alpha}}{\mathrm{d} t^{\alpha}}\mathbb{E}(N^{\alpha,\nu}(t))=(\lambda-\mu-\nu)\mathbb{E}(N^{\alpha,\nu}(t))
\end{equation}
with initial condition $\mathbb{E}(N^{\alpha,\nu}(0))=1$. By solving \eqref{expDE}, the required result is obtained.
\end{proof}

\begin{remark}
On substituting $\nu=0$ in \eqref{expt}, the expectation of time-changed LBDPC reduces to that of time-changed LBDP (see Orsingher and Polito (2011), Eq. (4.3)).
\end{remark}

\begin{figure}
	\centering
	\includegraphics[width=1\textwidth]{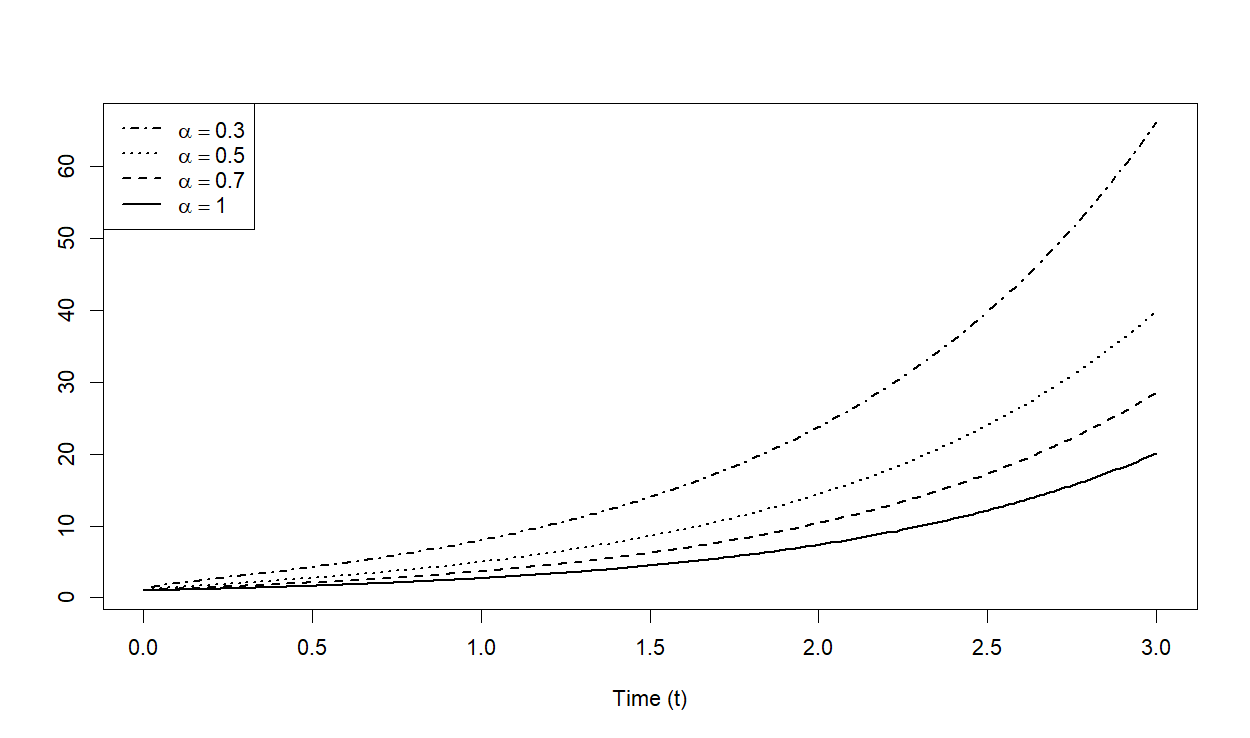}
	\caption{Plots of $\mathbb{E}(N^{\alpha,\nu}(t))$ for $\lambda=3$, $\mu=1$, $\nu=1$, and for different values of $\alpha$.}
	\label{fig2} 
\end{figure}

\begin{figure}
	\centering
	\includegraphics[width=1\textwidth]{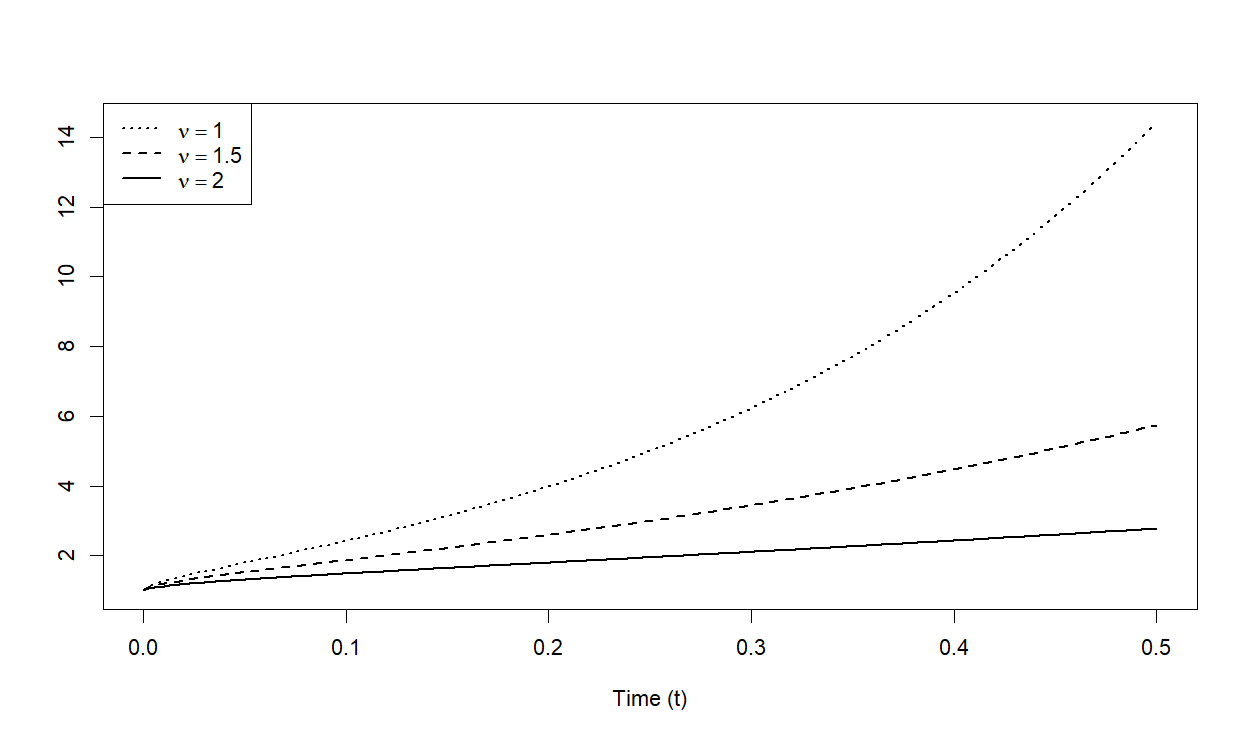}
	\caption{Plots of $\mathbb{E}(N^{\alpha,\nu}(t))$ for $\alpha=0.5$, $\lambda=4$, $\mu=1$, and for different values of $\nu$.}
	\label{fig3} 
\end{figure}

In Figure \ref{fig2}, the plot for expected population in time-changed LBDPC for different values of $\alpha$ is given. It is observed that the average growth in time-changed LBDPC increases as $\alpha$ decreases. In Figure \ref{fig3}, the plot for expected population in time-changed LBDPC for different values of $\nu$ is given. It is observed that the average growth in time-changed LBDPC increases with decrease in the catastrophe rate $\nu$.

\begin{theorem}\label{thm4.6}
For $\lambda\neq\mu$ and $h^{\alpha,\nu}_{1,1}(0)=1$, the variance of time-changed LBDPC is given by
{\small\begin{equation*}
	Var(N^{\alpha,\nu}(t))=\frac{2\lambda}{\lambda-\mu}E_{\alpha}((2\lambda-2\mu-\nu)t^{\alpha})-\Big(\frac{\lambda+\mu}{\lambda-\mu}\Big)E_{\alpha}((\lambda-\mu-\nu)t^{\alpha})-\big(E_{\alpha}((\lambda-\mu-\nu)t^{\alpha})\big)^{2}.
\end{equation*}}
\end{theorem}
\begin{proof}
Let us denote $u_{2}(t)$ as the following:
\begin{equation}\label{u2pgf}
	u_{2}(t)=\mathbb{E}\big(N^{\alpha,\nu}(t)(N^{\alpha,\nu}(t)-1)\big)=\frac{\partial^{2}}{\partial x^{2}}G_{1}^{\alpha,\nu}(x,t)\bigg|_{x=1}.
\end{equation}
On taking the second order derivative of \eqref{pgfDE} with respect to variable $x$, and using \eqref{exppgf} and \eqref{u2pgf}, we get
\begin{equation}\label{u2DE}
	\frac{\mathrm{d}^{\alpha}}{\mathrm{d}t^{\alpha}}u_{2}(t)=(2\lambda-2\mu-\nu)u_{2}(t)+2\lambda\mathbb{E}(N^{\alpha,\nu}(t))
\end{equation}
with initial condition $u_{2}(0)=0$. By taking the Laplace transform of \eqref{u2DE} and using \eqref{caputolp}, we obtain
\begin{align}\label{u2z}
	\tilde{u}_{2}(z)&=\frac{2\lambda z^{\alpha-1}}{(z^\alpha-2\lambda+2\mu+\nu)(z^\alpha-\lambda+\mu+\nu)}\nonumber\\
	&=\frac{2\lambda}{\lambda-\mu}\Big(\frac{z^{\alpha-1}}{z^{\alpha}-2\lambda+2\mu+\nu}-\frac{z^{\alpha-1}}{z^{\alpha}-\lambda+\mu+\nu}\Big),\,\lambda\neq\mu.
\end{align}
On taking the inverse Laplace transform of \eqref{u2z} and by using \eqref{ltm12}, we have
\begin{equation}\label{u2t}
	u_{2}(t)=\frac{2\lambda}{\lambda-\mu}\big(E_{\alpha}((2\lambda-2\mu-\nu)t^{\alpha})-E_{\alpha}((\lambda-\mu-\nu)t^{\alpha})\big),\,\lambda\neq\mu.
\end{equation}
By using \eqref{u2t}, we get the second moment of $\{N^{\alpha,\nu}(t)\}_{t\geq0}$ in the following form:
\begin{equation}\label{2ndmomentN}
	\mathbb{E}((N^{\alpha,\nu}(t))^{2})=\frac{2\lambda}{\lambda-\mu}E_{\alpha}((2\lambda-2\mu-\nu)t^{\alpha})-\Big(\frac{\lambda+\mu}{\lambda-\mu}\Big)E_{\alpha}((\lambda-\mu-\nu)t^{\alpha}),\,\lambda\neq\mu.
\end{equation}
On using \eqref{expt} and \eqref{2ndmomentN}, we get the required result.
\end{proof}

\begin{remark}
	On substituting $\nu=0$ in Theorem \ref{thm4.6}, the variance of time-changed LBDPC reduces to that of time-changed LBDP (see Orsingher and Polito (2011), Eq. (4.11)) for the case $\lambda\neq\mu$.
\end{remark}

\begin{remark}
For $\lambda=\mu$, the Laplace transform of $u_{2}(t)$ is
\begin{equation*}
	\tilde{u}_{2}(z)=\frac{2\lambda z^{\alpha-1}}{(z^{\alpha}+\nu)^{2}},
\end{equation*}
which on taking the inverse Laplace transform yields 
\begin{equation}\label{u2=}
	u_{2}(t)=2\lambda t^{\alpha}E_{\alpha,\alpha+1}^{2}(-\nu t^{\alpha}),
\end{equation}
where we have used \eqref{ltm12}. Now, by using \eqref{expt}, \eqref{u2pgf} and \eqref{u2=}, the variance of time-changed LBDPC is given by
\begin{equation*}
	Var(N^{\alpha,\nu}(t))=2\lambda t^{\alpha}E_{\alpha,\alpha+1}^{2}(-\nu t^{\alpha})+E_{\alpha}(-\nu t^{\alpha})-\big(E_{\alpha}(-\nu t^{\alpha})\big)^{2},
\end{equation*}
which reduces to the variance of time-changed LBDP on substituting $\nu=0$ (see Orsingher and Polito (2011), Eq. (4.13)).
\end{remark}

Some numerical examples for the expectation and variance of $N^{\alpha,\nu}(1)$ are given in Table \ref{tab1}.

\begin{table}[h]
	\centering
	\begin{tabular}{c c c c c}
		\hline
		$\lambda$ & $\mu$ & $\nu$ & $\mathbb{E}(N^{\alpha,\nu}(1))$ & $\mathrm{Var}(N^{\alpha,\nu}(1))$ \\ 
		\hline
		0.5 & 0.3 & 0.2 & 1.0000 & \ \ 1.3630 \\ 
		1.0 & 0.1 & 0.3 & 2.2989 & 33.3586 \\ 
		1.5 & 1.1 & 1.2 & 0.4891 & \ \ 1.6125 \\
		2.0 & 1.6 & 1.7 & 0.3576 & \ \ 1.2186 \\
		2.5 & 2.1 & 2.2 & 0.2786 & \ \ 0.9533 \\
		3.0 & 2.0 & 1.5 & 0.6157 & \ \ 8.2566 \\
		3.5 & 2.5 & 1.3 & 0.7346 & 14.2237 \\
		4.0 & 2.7 & 2.6 & 0.3576 & \ \ 4.1827 \\
		4.5 & 2.8 & 2.9 & 0.3785 & \ \ 8.5673 \\
		5.0 & 2.6 & 3.5 & 0.4017 & 42.2386 \\
		\hline
	\end{tabular}
	\vspace{0.5cm}
	\caption{Numerical examples of expectation and variance of time-changed LBDPC.}
	\label{tab1}
\end{table}


\subsection{Sample path simulation of time-changed LBDPC}
We give the sample paths simulation of time-changed LBDPC by using the modified Gillespie algorithm which was introduced by Cahoy {\it et al.} (2015) and Ascione {\it et al.} (2020). From Theorem \ref{thmSojourn}, the sojourn time at non-zero state $m$ is Mittag-Leffler distributed with parameters $\alpha$ and $m\lambda+m\mu+\nu$. We use Corollary 7.3 of Asione {\it et al.} (2020) to simulate the Mittag-Leffler distributed sojourn times in time-changed LBDPC, where zero is the absorbing state. 
The algorithm to simulate time-changed LBDPC is given as follows:

\noindent Step 1: Initialize the process by fixing $N^{\alpha,\nu}(0)=m\geq1$ and $Y_{0}=0$.

\noindent Step 2: Suppose a process has been simulated till the $k$th event that happens in time $Y_{k}$ and current state is $N^{\alpha,\nu}(Y_{k})=n$. If $n\geq1$ then simulate random variable $I_{n}^{\alpha}\overset{d}{=}X_{n}^{1/\alpha}T_{\alpha}$. Here, $X_{n}$ is an exponential random variable with rate $n\lambda+n\mu+\nu$ and $T_{\alpha}\overset{d}{=}D_{\alpha}(1)$ such that $X_{n}$ and $D_{\alpha}(1)$ are independent.

\noindent Step 3: If $n=0$ then fix $N^{\alpha,\nu}(t)=0$ for all $t>Y_{k}$ else get the next event time as $Y_{k+1}=Y_{k}+I_n$.

\noindent Step 4: If $n\geq1$ then simulate the random variable $U\overset{d}{=}Uni(0,1)$. If $U<\frac{n\lambda}{n\lambda+n\mu+\nu}$, then set $N^{\alpha,\nu}(Y_{k+1})=n+1$, if $U<\frac{n\lambda+n\mu}{n\lambda+n\mu+\nu}$, then set $N^{\alpha,\nu}(Y_{k+1})=n-1$ otherwise $N^{\alpha,\nu}(Y_{k+1})=0$. 

\noindent Step 5: Obtain the whole process by setting $N^{\alpha,\nu}(t)=N^{\alpha,\nu}(Y_{k})$ for all $t\in[Y_{k},Y_{k+1})$.

\begin{figure}
	\centering
	\includegraphics[width=1\textwidth]{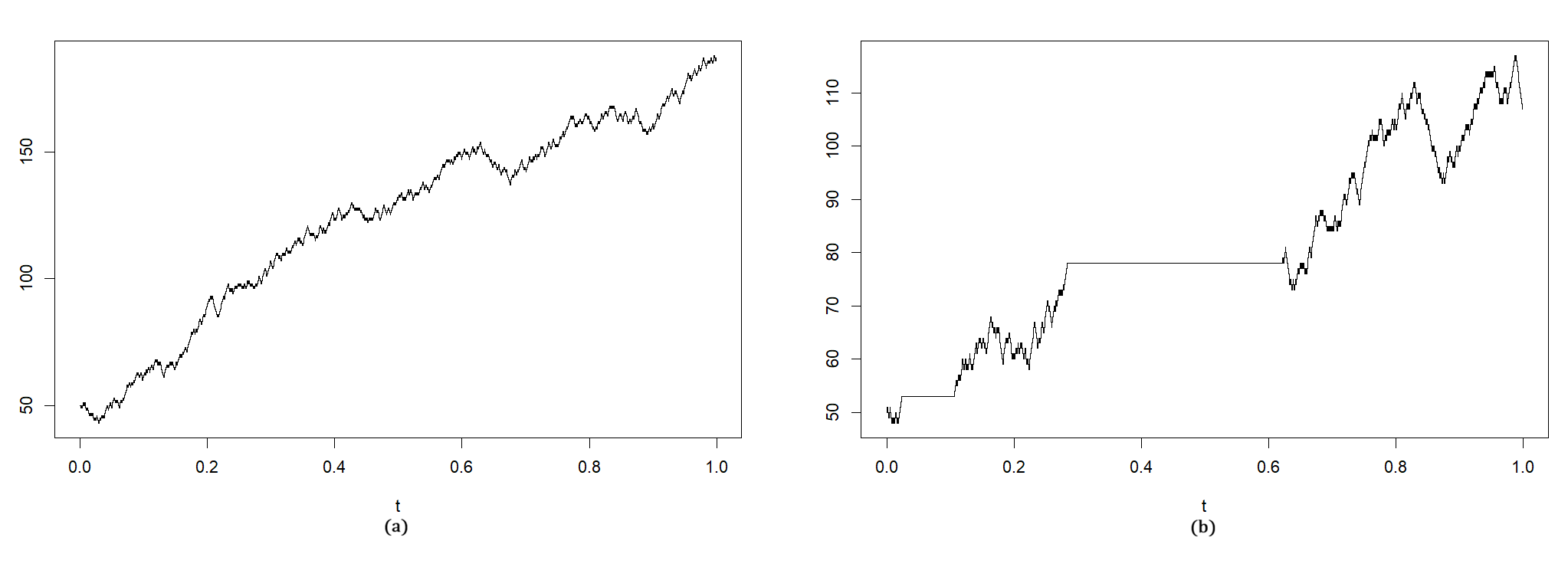}
	\caption{Sample path simulation of time-changed LBDPC. Plot (a) is for parameters $\alpha=1$, $\lambda=15$, $\mu=11$ and $\nu=2$ and Plot (b) is for parameters $\alpha=0.5$, $\lambda=15$, $\mu=11$ and $\nu=2$.}
	\label{fig4} 
\end{figure}

\begin{figure}
	\centering
	\includegraphics[width=1\textwidth]{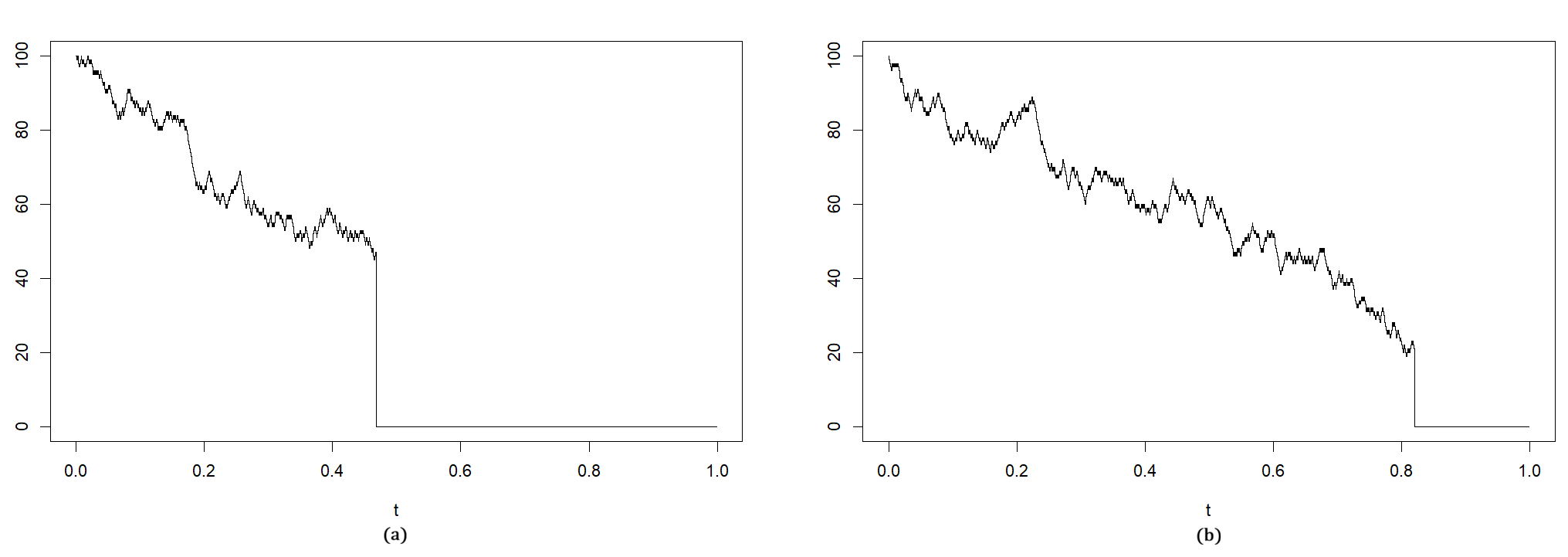}
	\caption{Sample path simulation of time-changed LBDPC. Plot (a) is for parameters $\alpha=0.3$, $\lambda=10$, $\mu=12$ and $\nu=3$ and Plot (b) is for parameters $\alpha=0.8$, $\lambda=10$, $\mu=12$ and $\nu=3$.}
	\label{fig5} 
\end{figure}

\section{Tempered birth-death process with catastrophe}\label{SecTempBDPC}
In this section, we introduce and study the tempered birth-death process with catastrophes, in short, the tempered BDPC. It is obtained by time-changing BDPC with the first hitting time of tempered stable subordinator.
 
Let $\{\mathcal{N}^{\nu}(t)\}_{t\geq0}$ be a BDPC with state space $I=\{0,1,2,\dots\}$. Here, $\lambda_{n}$'s for $n=1,2,\dots$ are state dependent positive birth rates with $\lambda_{0}\geq0$, $\mu_{n}$'s for $n=1,2,3,\dots$ are state dependent positive death rates and $\nu$ is the rate of catastrophe at any state. Let  $\{D_{\theta,\alpha}(t)\}_{t\geq0}$ be a tempered stable subordinator with $\theta>0$ being its tempering parameter and $0<\alpha<1$ being its stability index, for definition see Section \ref{SubSecSubordinator}. Also, let $\{Y_{\theta,\alpha}(t)\}_{t\geq0}$ with $Y_{\theta,\alpha}(0)=0$ be its first hitting time such that it is independent of the BDPC $\{\mathcal{N}^{\nu}(t)\}_{t\geq0}$. Then, we define the tempered BDPC $\{\mathscr{N}^{\theta,\alpha,\nu}(t)\}_{t\geq0}$, $0<\alpha\leq1$, $\nu\geq0$ as the following time-changed process: 
\begin{equation*}
	\mathscr{N}^{\theta,\alpha,\nu}(t)\coloneqq\mathcal{N^{\nu}}(Y_{\theta,\alpha}(t)),\,\theta>0,\,0<\alpha<1,
\end{equation*}
with $\mathscr{N}^{0,1,\nu}(t)=\mathcal{N}^{\nu}(t)$, $t\geq0$. 

Similarly, the tempered birth-death process without catastrophe, that is, the tempered BDP $\{\mathscr{N}^{\theta,\alpha}(t)\}_{t\geq0}$, $\theta>0$, $0<\alpha\leq1$ is defined as
\begin{equation*}
	\mathscr{N}^{\theta,\alpha}(t)\coloneqq\mathcal{N}(Y_{\theta,\alpha}(t)),\,\theta>0,\,0<\alpha<1,
\end{equation*}
with $\mathscr{N}^{0,1}(t)=\mathcal{N}(t)$, $t\geq0$. 

For $\nu=0$, the tempered BDPC reduces to the tempered BDP.

For $m,n\in I$, let us denote the state probabilities of $\{\mathscr{N}^{\theta,\alpha,\nu}(t)\}_{t\geq0}$ and $\{\mathscr{N}^{\theta,\alpha}(t)\}_{t\geq0}$ by
\begin{equation*}
	\mathrm{p}_{m,n}^{\theta,\alpha,\nu}(t)=\mathrm{Pr}\{ \mathscr{N}^{\theta,\alpha,\nu}(t)=n|\mathscr{N}^{\theta,\alpha,\nu}(0)=m\}
\end{equation*}
and
\begin{equation*}
	\mathrm{p}_{m,n}^{\theta,\alpha}(t)=\mathrm{Pr}\{\mathscr{ N}^{\theta,\alpha}(t)=n|\mathscr{N}^{\theta,\alpha}(0)=m\},
\end{equation*} 
respectively.

\begin{theorem}\label{thm5.1}
	The governing system of differential equations for the state probabilities of tempered BDPC is given by
	\begin{align*}\label{DE8}
		\frac{\mathrm{d}^{\theta,\alpha}}{\mathrm{d}t^{\theta,\alpha}} \mathrm{p}_{m,0}^{\theta,\alpha,\nu}(t)&=-(\lambda_{0}+\nu)\mathrm{p}_{m,0}^{\theta,\alpha,\nu}(t)+\mu_{1}\mathrm{p}_{m,1}^{\theta,\alpha,\nu}(t)+\nu,\\
		\frac{\mathrm{d}^{\theta,\alpha}}{\mathrm{d}t^{\theta,\alpha}}\mathrm{p}_{m,n}^{\theta,\alpha,\nu}(t)&=-(\lambda_{n}+\mu_{n}+\nu)\mathrm{p}_{m,n}^{\theta,\alpha,\nu}(t)+\lambda_{n-1}\mathrm{p}_{m,n-1}^{\theta,\alpha,\nu}(t)+\mu_{n+1}\mathrm{p}_{m,n+1}^{\theta,\alpha,\nu}(t),\,n=1,2,\dots
	\end{align*}
	with initial condition $\mathrm{p}_{m,n}^{\theta,\alpha,\nu}(0)=\delta_{m,n}$.
	Here, $\frac{\mathrm{d}^{\theta,\alpha}}{\mathrm{d}t^{\theta,\alpha}}$ is the Caputo tempered fractional derivative as defined in Section \ref{SubSecFractionalDerivative}. 
\end{theorem}
\begin{proof}
For $m,n\in I$ and $t\geq0$, we have
\begin{equation}\label{pmntemp}
	\mathrm{p}_{m,n}^{\theta,\alpha,\nu}(t)=\mathrm{Pr}\{\mathcal{N}^{\nu}(Y_{\theta,\alpha}(t))=n|\mathcal{N}^{\nu}(0)=m\}
	=\int_{0}^{\infty}p_{m,n}^{\nu}(y)\mathrm{Pr}\{Y_{\theta,\alpha}(t)\in\mathrm{d}y\}.
\end{equation}
On taking the Laplace transform of \eqref{pmntemp} and by using \eqref{temperedPDFLT}, we obtain
\begin{align}\label{pmnztemp}
	\tilde{\mathrm{p}}_{m,n}^{\theta,\alpha,\nu}(z)&=\frac{\phi_{\theta,\alpha}(z)}{z}\int_{0}^{\infty}p_{m,n}^{\nu}(y)e^{-y\phi_{\theta,\alpha}(z)}\mathrm{d}y
	=\frac{\phi_{\theta,\alpha}(z)}{z}\tilde{p}_{m,n}^{\nu}(\phi_{\theta,\alpha}(z)),\, z>0.
\end{align}
By using $\mathrm{p}_{m,n}^{\theta,\alpha,\nu}(0)=p_{m,n}^{\nu}(0)$ and \eqref{pmnztemp} in \eqref{ltDE1}, we get
{\small\begin{align*}\label{ltDE2temp}
	\phi_{\theta,\alpha}(z)\tilde{\mathrm{p}}_{m,0}^{\theta,\alpha,\nu}(z)-\frac{\phi_{\theta,\alpha}(z)}{z}\mathrm{p}_{m,0}^{\theta,\alpha,\nu}(0)&=-(\lambda_{0}+\nu)\tilde{\mathrm{p}}_{m,0}^{\theta,\alpha,\nu}(z)+\mu_{1}\tilde{\mathrm{p}}_{m,1}^{\theta,\alpha,\nu}(z)+\frac{\nu}{z},\\
	\phi_{\theta,\alpha}(z)\tilde{\mathrm{p}}_{m,n}^{\theta,\alpha,\nu}(z)-\frac{\phi_{\theta,\alpha}(z)}{z}\mathrm{p}_{m,n}^{\theta,\alpha,\nu}(0)&=-(\lambda_{n}+\mu_{n}+\nu)\tilde{\mathrm{p}}_{m,n}^{\theta,\alpha,\nu}(z)+\lambda_{n-1}\tilde{\mathrm{p}}_{m,n-1}^{\theta,\alpha,\nu}(z)+\mu_{n+1}\tilde{\mathrm{p}}_{m,n+1}^{\theta,\alpha,\nu}(z)
\end{align*}}
whose inverse Laplace transform yields the required result by using \eqref{CaputotemperedLT}.
\end{proof}


\begin{proposition}\label{thm5.2}
The following relationship holds between the Laplace transforms of state probabilities of tempered BDPC and tempered BDP:
\begin{equation*}
	\tilde{\mathrm{p}}_{m,n}^{\theta,\alpha,\nu}(z)=\psi_{\theta,\alpha}^{\nu}(z)\Big(\frac{\phi_{\theta,\alpha}(z)}{z}\tilde{\mathrm{p}}_{m,n}^{\theta,\alpha}\big(((z+\theta)^{\alpha}+\nu)^{1/\alpha}-\theta\big)+\frac{\nu}{z}\tilde{\mathrm{p}}_{0,n}^{\theta,\alpha}\big(((z+\theta)^{\alpha}+\nu)^{1/\alpha}-\theta\big)\Big),
\end{equation*}
where 
\begin{equation}\label{psiEq}
	\psi_{\theta,\alpha}^{\nu}(z)=\frac{((z+\theta)^{\alpha}+\nu)^{1/\alpha}-\theta}{(z+\theta)^{\alpha}+\nu-\theta^{\alpha}}.
\end{equation}
\end{proposition}
\begin{proof}
For $m,n\in I$ and $z>0$, we have
\begin{equation}\label{pmntempp}
	\tilde{\mathrm{p}}_{m,n}^{\theta,\alpha,\nu}(z)=\frac{\phi_{\theta,\alpha}(z)}{z}\Big(\tilde{p}_{m,n}(\phi_{\theta,\alpha}(z)+\nu)+\frac{\nu}{\phi_{\theta,\alpha}(z)}\tilde{p}_{0,n}(\phi_{\theta,\alpha}(z)+\nu)\Big),
\end{equation}
which follows by using Eq. (3) of Di Crescenzo {\it et al.} (2008) in \eqref{pmnztemp}. 

On taking the Laplace transform of $\mathrm{p}^{\theta,\alpha}_{m,n}(t)$ and by using \eqref{temperedPDFLT}, we obtain
\begin{equation}\label{pmntempz}
	\tilde{p}_{m,n}(z)=\frac{(z+\theta^\alpha)^{1/\alpha}-\theta}{z}\tilde{\mathrm{p}}_{m,n}^{\theta,\alpha}((z+\theta^{\alpha})^{1/\alpha}-\theta).
\end{equation}
By using $\phi_{\theta,\alpha}(z)=(z+\theta)^{\alpha}-\theta^{\alpha}$ and \eqref{pmntempz} in \eqref{pmntempp}, we get the required result.
\end{proof}



\begin{theorem}
In tempered BDPC, the inter-occurrence times of catastrophes $C_{\theta,\alpha}^{\nu}$ are tempered Mittag-Leffler distributed random variables. That is,
\begin{equation*}
	\mathrm{Pr}\{C_{\theta,\alpha}^{\nu}>t\}=e^{-\theta t}\sum_{n=0}^{\infty}\sum_{r=0}^{\infty}(-\nu t^{\alpha})^{n}(\theta t)^{r}E_{\alpha,n\alpha+r+1}^{n}(\theta^{\alpha}t^{\alpha}).
\end{equation*}
\end{theorem}
\begin{proof}
Let $\{X(t)\}_{t\geq0}$ be the modified process as defined in the proof of Theorem \ref{thm3.2}. We define its time-changed variant $\{X^{\theta,\alpha}(t)\}_{t\geq0}$ as follows: $X^{\theta,\alpha}(t)\coloneqq X(Y_{\theta,\alpha}(t))$, $t\geq0$, where $\{Y_{\theta,\alpha}(t)\}_{t\geq0}$ is the inverse tempered stable subordinator which is independent of $\{X(t)\}_{t\geq0}$. Its state probabilities are denoted by $x^{\theta,\alpha}_{n}(t)=\mathrm{Pr}\{X^{\theta,\alpha}(t)=n|X^{\theta,\alpha}(0)=m\}$, $n\in\{-1,k\}$ whose Laplace transform is given by
\begin{equation}\label{xmztemp}
	\tilde{x}_{n}^{\theta,\alpha}(z)
	=\frac{\phi_{\theta,\alpha}(z)}{z}\int_{0}^{\infty}e^{-y\phi_{\theta,\alpha}(z)}x_{n}(y)\mathrm{d}y
	=\frac{\phi_{\theta,\alpha}(z)}{z}\tilde{x}_{n}(\phi_{\theta,\alpha}(z)),\, z>0,
\end{equation}
which is obtained by using \eqref{temperedPDFLT}.

On taking the Laplace transform of \eqref{DE3}, and by using \eqref{xmztemp} and \eqref{CaputotemperedLT}, we obtain
\begin{align}\label{DExkTemp}
	\frac{\mathrm{d}^{\theta,\alpha}}{\mathrm{d}t^{\theta,\alpha}}x_{-1}^{\theta,\alpha}(t)&=\nu x_{k}^{\theta,\alpha}(t),\nonumber\\
	\frac{\mathrm{d}^{\theta,\alpha}}{\mathrm{d}t^{\theta,\alpha}}x_{k}^{\theta,\alpha}(t)&=-\nu x_{k}^{\theta,\alpha}(t)
\end{align}
with initial condition $x_{k}^{\theta,\alpha}(0)=1$. The Laplace transform of \eqref{DExkTemp} is given by
\begin{align}\label{xkLT}
	\tilde{x}_{k}^{\theta,\alpha}(z)
	&=\frac{\phi_{\theta,\alpha}(z)}{z(\phi_{\theta,\alpha}(z)+\nu)}\nonumber\\
	&=\frac{1}{z}\Big(1+\frac{\nu}{(z+\theta)^{\alpha}-\theta^{\alpha}}\Big)^{-1}\nonumber\\
	&=\sum_{n=0}^{\infty}\frac{(-\nu)^n}{z((z+\theta)^{\alpha}-\theta^{\alpha})^n}.
\end{align}
On taking the inverse Laplace transform of \eqref{xkLT}, we have
\begin{align*}\label{xkthetaalpha}
	x_{k}^{\theta,\alpha}(t)
	&=\sum_{n=0}^{\infty}(-\nu)^{n}\int_{0}^{t}e^{-\theta y}y^{n\alpha-1}E_{\alpha,n\alpha}^{n}(\theta^{\alpha}y^{\alpha})\mathrm{d}y\nonumber\\
	&=e^{-\theta t}\sum_{n=0}^{\infty}\sum_{r=0}^{\infty}(-\nu t^{\alpha})^{n}(\theta t)^{r}E_{\alpha,n\alpha+r+1}^{n}(\theta^{\alpha}t^{\alpha}),
\end{align*}
where the second equality is obtained by using the following result (see Kilbas {\it et al.} (2004), Eq. (2.26)):
\begin{equation*}\label{KilbasIntegrl}
	\int_{0}^{x}u^{\mu-1}E_{\rho,\mu}^{\gamma}(\omega u^{\rho})(x-u)^{\nu-1}\mathrm{d}u=\Gamma(\nu)x^{\mu+\nu-1}E_{\rho,\mu+\nu}^{\gamma}(\omega x^{\rho}).
\end{equation*}

Let $C_{\theta,\alpha}^{\nu}$ be the occurrence time of the first catastrophe in $\{\mathscr{ N}^{\theta,\alpha,\nu}(t)\}_{t\geq0}$ such that $\mathscr{N}^{\theta,\alpha,\nu}(0)=m$. Then,
\begin{align*}
	F_{C_{\theta,\alpha}^{\nu}}(t)\coloneqq\mathrm{Pr}\{C_{\theta,\alpha}^{\nu}\leq t\}&=1-\mathrm{Pr}\{C_{\theta,\alpha}^{\nu}>t\}\nonumber\\
	&=1-\mathrm{Pr}\{X^{\theta,\alpha}(t)=k|X^{\theta,\alpha}(0)=k\}\nonumber\\
	&=1-x_{k}^{\theta,\alpha}(t).
\end{align*}
As $\{\mathcal{ N}^{\nu}(t)\}_{t\geq0}$ is a Markov process, $\{\mathscr{ N}^{\theta,\alpha,\nu}(t)\}_{t\geq0}$ is a semi-Markov process. Let $I_n$ be its $n$th jump time. Then, $\mathscr{ N}^{\theta,\alpha,\nu}(I_n)$ is a time-homogeneous Markov chain (see Gikhman and Skorokhod (1975)). Thus, we have the required result.
\end{proof}

\begin{theorem}
For any state $m>0$, the sojourn time $S^{\theta,\alpha}_{m}$ for tempered BDPC has the following tempered Mittag-Leffler distribution:
\begin{equation*}
	\mathrm{Pr}\{S^{\theta,\alpha}_{m}>t\}=e^{-\theta t}\sum_{n=0}^{\infty}\sum_{r=0}^{\infty}(-(\lambda_{m}+\mu_{m}+\nu) t^{\alpha})^{n}(\theta t)^{r}E_{\alpha,n\alpha+r+1}^{n}(\theta^{\alpha}t^{\alpha}).
\end{equation*}
\end{theorem}

\begin{lemma}\label{lemma5.1}
Let $\{D_{\theta,\alpha}(t)\}_{t\geq0}$ be a tempered stable subordinator with $\theta>0$ being its tempering parameter and $0<\alpha<1$ be its stability index. Also, let $T$ be a random variable having exponential distribution with rate $\lambda$ such that it is independent of $\{D_{\theta,\alpha}(t)\}_{t\geq0}$. Then, $D_{\theta,\alpha}(T)$ has tempered Mittag-Leffler distribution given by
\begin{equation*}
	\mathrm{Pr}\{D_{\theta,\alpha}(T)>t\}=e^{-\theta t}\sum_{n=0}^{\infty}\sum_{r=0}^{\infty}(-\lambda t^{\alpha})^{n}(\theta t)^{r}E_{\alpha,n\alpha+r+1}^{n}(\theta^{\alpha}t^{\alpha}).
\end{equation*}
\end{lemma}
\begin{proof}
Consider
\begin{align}\label{TempMLdistEq}
	\mathrm{Pr}\{D_{\theta,\alpha}(T)>t\}
	&=\mathrm{Pr}\{T>Y_{\theta,\alpha}(t)\}\nonumber\\
	&=\int_{0}^{\infty}\mathrm{Pr}\{T>y\}\mathrm{Pr}\{Y_{\theta,\alpha}(t)\in\mathrm{d}y\}\nonumber\\
	&=\int_{0}^{\infty}e^{-\lambda y}\mathrm{Pr}\{Y_{\theta,\alpha}(t)\in\mathrm{d}y\}.
\end{align}
The result holds as the Laplace transform of \eqref{TempMLdistEq} is equal to \eqref{xkLT}.
\end{proof}

Let us define the first visit time to state zero for tempered BDPC as follows:
\begin{equation*}
	T_{m,0}^{\theta,\alpha,\nu}=\inf\{t\geq0:\mathscr{ N}^{\theta,\alpha,\nu}(t)=0\},\,\mathscr{ N}^{\theta,\alpha,\nu}(0)=m>0.
\end{equation*}
Also, let $T_{m,0}^{\theta,\alpha,0}$ be the first visit time to state zero for the corresponding tempered BDP $\{\mathscr{ N}^{\theta,\alpha}(t)\}_{t\geq0}$.

\begin{theorem}\label{thm5.3}
Let $Y$ be a random variable which is equal in distribution with the first occurrence time of catastrophe in tempered BDPC such that $T_{m,0}^{\theta,\alpha,0}$ and $Y$ are conditionally independent given $Y_{\theta,\alpha}(t)$ for all $t>0$. If zero is an absorbing state of the tempered BDPC, that is, $\lambda_{0}=0$ and $\mathscr{ N}^{\theta,\alpha,\nu}(0)=m>0$, we have 
\begin{equation}\label{TTtemp}
	T_{m,0}^{\theta,\alpha,\nu}\overset{d}{=} \min\{T_{m,0}^{\theta,\alpha,0},Y\}.
\end{equation}
\end{theorem}
\begin{proof}
Recall that $T_{m,0}^{\nu}$ and $T_{m,0}$ are the first visit time to state zero for BDPC $\{\mathcal{ N}^{\nu}(t)\}_{t\geq0}$ and BDP $\{\mathcal{ N}(t)\}_{t\geq0}$, respectively. 

As zero is an absorbing state of tempered BDPC, we have
\begin{align}\label{survTmoalphanutheta}
		\mathrm{Pr}\{T_{m,0}^{\theta,\alpha,\nu}>t\}
		&=\int_{0}^{\infty}e^{-\nu y}\mathrm{Pr}\{T_{m,0}>y\}\mathrm{Pr}\{Y_{\theta,\alpha}(t)\in\mathrm{d}y\},
\end{align}
which is obtained by using $T_{m,0}^{\nu}\overset{d}{=}\min\{T_{m,0},X\}$ such that $X$ is an exponential random variable with rate $\nu$ (see Di Crescenzo {\it et al.} (2008), p. 2250).
Consider
{\small\begin{align}\label{thetatemp}
	\mathrm{Pr}\{\min\{T_{m,0}^{\theta,\alpha,0},Y\}>t\}
	&=\mathrm{Pr}\{T_{m,0}^{\theta,\alpha,0}>t,Y>t\}\nonumber\\
	&=\int_{0}^{\infty}\mathrm{Pr}\{T_{m,0}^{\theta,\alpha,0}>t,Y>t|Y_{\theta,\alpha}(t)\in\mathrm{d}y\}\mathrm{Pr}\{Y_{\theta,\alpha}(t)\in\mathrm{d}y\}\nonumber\\
	&=\int_{0}^{\infty}\mathrm{Pr}\{T_{m,0}^{\theta,\alpha,0}>t|Y_{\theta,\alpha}(t)\in\mathrm{d}y\}\mathrm{Pr}\{Y>t|Y_{\theta,\alpha}(t)\in\mathrm{d}y\}\mathrm{Pr}\{Y_{\theta,\alpha}(t)\in\mathrm{d}y\},
\end{align}}
where the last step is obtained by using the conditional independence of $T_{m,0}^{\theta,\alpha,0}$ and $Y$ given $Y_{\theta,\alpha}(t)$ for all $t>0$.
As $\mathcal{N}(y)$ and $Y_{\theta,\alpha}(t)$ are independent, we have
\begin{equation}\label{Tmoalphatheta}
		\mathrm{Pr}\{T_{m,0}^{\theta,\alpha,0}>t|Y_{\theta,\alpha}(t)\in\mathrm{d}y\}
		=\mathrm{Pr}\{T_{m,0}>y\}.
\end{equation}
Let $X$ be independent of the tempered stable subordinator $\{D_{\theta,\alpha}(t)\}_{t\geq0}$, $\theta>0$, $0<\alpha<1$. On using the Lemma \ref{lemma5.1}, we get
\begin{equation}\label{Xtemp}
	\mathrm{Pr}\{Y>t|Y_{\theta,\alpha}(t)\in\mathrm{d}y\}
	=e^{-\nu y}.
\end{equation}
Now, by using \eqref{Tmoalphatheta} and \eqref{Xtemp} in \eqref{thetatemp}, we obtain
\begin{align*}
	\mathrm{Pr}\{\min\{T_{m,0}^{\theta,\alpha,0},Y\}>t\}
	&=\int_{0}^{\infty}e^{-\nu y}\mathrm{Pr}\{T_{m,0}>y\}\mathrm{Pr}\{Y_{\theta,\alpha}(t)\in\mathrm{d}y\}\nonumber\\
	&=\mathrm{Pr}\{T_{m,0}^{\theta,\alpha,\nu}>t\}.
\end{align*}
This completes the proof.
\end{proof}

\begin{corollary}
Let zero be an absorbing state of tempered BDPC. The Laplace transform of the pdf of first visit time to state zero of tempered BDPC is given by
\begin{equation*}
	\tilde{f}_{m,0}^{\theta,\alpha,\nu}(z)=\frac{\nu}{\nu+\phi_{\theta,\alpha}(z)}+\frac{\phi_{\theta,\alpha}(z)}{\nu+\phi_{\theta,\alpha}(z)}\tilde{f}_{m,0}(\nu+\phi_{\theta,\alpha}(z)),
\end{equation*}
where $f_{m,0}^{\theta,\alpha,\nu}(t)$ is the pdf of $T_{m,0}^{\theta,\alpha,\nu}$.
\end{corollary}
\begin{proof}
From \eqref{survTmoalphanutheta}, we have
\begin{align}\label{fmoalphanutheta}
	f_{m,0}^{\theta,\alpha,\nu}(t)
	&=-\int_{0}^{\infty}e^{-\nu y}\mathrm{Pr}\{T_{m,0}>y\}\frac{\mathrm{d}}{\mathrm{d}t}\mathrm{Pr}\{Y_{\theta,\alpha}(t)\in\mathrm{d}y\}.
\end{align}
On taking the Laplace transform of \eqref{fmoalphanutheta} and using \eqref{temperedPDFLT}, we obtain
\begin{align*}
	\tilde{f}_{m,0}^{\theta,\alpha,\nu}(z)
	&=-\int_{0}^{\infty}e^{-\nu y}\mathrm{Pr}\{T_{m,0}>y\}(\phi_{\theta,\alpha}(z)e^{-\phi_{\theta,\alpha}(z)y}\mathrm{d}y-\mathrm{Pr}\{Y_{\theta,\alpha}(0)\in\mathrm{d}y\})\nonumber\\
	&=1-\phi_{\theta,\alpha}(z)\int_{0}^{\infty}\int_{y}^{\infty}e^{-(\nu+\phi_{\theta,\alpha}(z))y}f_{m,0}(x)\mathrm{d}x\mathrm{d}y\nonumber\\
	&=1-\phi_{\theta,\alpha}(z)\int_{0}^{\infty}f_{m,0}(x)\int_{0}^{x}e^{-(\nu+\phi_{\theta,\alpha}(z))y}\mathrm{d}y\mathrm{d}x\nonumber\\
	&=1+\frac{\phi_{\theta,\alpha}(z)}{\nu+\phi_{\theta,\alpha}(z)}(\tilde{f}_{m,0}(\nu+\phi_{\theta,\alpha}(z))-1).
\end{align*}
Thus, we get the required result.
\end{proof}

\begin{theorem}
Let zero be an absorbing state of tempered BDPC. If tempered BDPC starts at $m>0$, that is, $\mathscr{N}^{\theta,\alpha,\nu}(0)=m$. Then, the expected first visit time to state zero is given by
\begin{equation*}
	\mathbb{E}(T_{m,0}^{\theta,\alpha,\nu})=\frac{\alpha\theta^{\alpha-1}}{\nu}(1-\tilde{f}_{m,0}(\nu)).
\end{equation*}
\end{theorem}
\begin{proof}
	As $T_{m,0}^{\theta,\alpha,\nu}$ is a non-negative continuous random variable, we have
	\begin{align}\label{expTmo}
		\mathbb{E}(T_{m,0}^{\theta,\alpha,\nu})
		&=\int_{0}^{\infty}\mathrm{Pr}\{T_{m,0}^{\theta,\alpha,\nu}>t\}\mathrm{d}t\nonumber\\
		&=\int_{0}^{\infty}\int_{0}^{\infty}\mathrm{Pr}\{T_{m,0}^{\nu}>y\}\mathrm{Pr}\{Y_{\theta,\alpha}(t)\in\mathrm{d}y\}\mathrm{d}t\nonumber\\
		&=\int_{0}^{\infty}\mathrm{Pr}\{T_{m,0}^{\nu}>Y_{\theta,\alpha}(t)\}\mathrm{d}t\nonumber\\
		&=\int_{0}^{\infty}\mathrm{Pr}\{D_{\theta,\alpha}(T_{m,0}^{\nu})>t\}\mathrm{d}t\nonumber\\
		&=\mathbb{E}\big(D_{\theta,\alpha}(T_{m,0}^{\nu})\big)\nonumber\\
		&=\int_{0}^{\infty}\mathbb{E}(D_{\theta,\alpha}(x))\mathrm{Pr}\{T_{m,0}^{\nu}\in\mathrm{d}x\}\nonumber\\
		&=\alpha\theta^{\alpha-1}\int_{0}^{\infty}x\mathrm{Pr}\{T_{m,0}^{\nu}\in\mathrm{d}x\}\nonumber\\
		&=\alpha\theta^{\alpha-1}\mathbb{E}(T_{m,0}^{\nu}),
	\end{align}
	where we have used the expectation of tempered stable subordinator in the penultimate step (see Gupta {\it et al.} (2020), p. 948). Finally, the proof is complete on using the expectation of the first visit time to state zero for BDPC (see Di Crescenzo {\it et al.} (2008), Eq. (7)) in \eqref{expTmo}.
\end{proof}

\begin{theorem}
Let zero be an absorbing state of tempered BDPC.
For $\mathscr{N}^{\theta,\alpha,\nu}(0)=m$, the variance of the first visit time to state zero is given by
\begin{equation*}
	Var(T_{m,0}^{\theta,\alpha,\nu})=\frac{\alpha(1-\alpha)\theta^{\alpha-2}}{\nu}(1-\tilde{f}_{m,0}(\nu))+\frac{2\alpha^2\theta^{2(\alpha-1)}}{\nu}\frac{\mathrm{d}}{\mathrm{d}\nu}\tilde{f}_{m,0}(\nu)+\frac{\alpha^2\theta^{2(\alpha-1)}}{\nu^2}(1-(\tilde{f}_{m,0}(\nu))^2).
\end{equation*}
\end{theorem}
\begin{proof}
	As $T_{m,0}^{\theta,\alpha,\nu}$ is a non-negative continuous random variable, we have
	\begin{align}\label{2expTmo}
		\mathbb{E}((T_{m,0}^{\theta,\alpha,\nu})^2)
		&=2\int_{0}^{\infty}t\mathrm{Pr}\{T_{m,0}^{\theta,\alpha,\nu}>t\}\mathrm{d}t\nonumber\\
		&=2\int_{0}^{\infty}t\mathrm{Pr}\{T_{m,0}^{\nu}>Y_{\theta,\alpha}(t)\}\mathrm{d}t\nonumber\\
		&=2\int_{0}^{\infty}t\mathrm{Pr}\{D_{\theta,\alpha}(T_{m,0}^{\nu})>t\}\mathrm{d}t\nonumber\\
		&=\mathbb{E}((D_{\theta,\alpha}(T_{m,0}^{\nu}))^2)\nonumber\\
		&=\int_{0}^{\infty}\mathbb{E}((D_{\theta,\alpha}(x))^2)\mathrm{Pr}\{T_{m,0}^{\nu}\in\mathrm{d}x\}\nonumber\\
		&=\int_{0}^{\infty}\big(\alpha(1-\alpha)\theta^{\alpha-2}x+\alpha^2\theta^{2(\alpha-1)}x^2\big)\mathrm{Pr}\{T_{m,0}^{\nu}\in\mathrm{d}x\}\nonumber\\
		&=\alpha(1-\alpha)\theta^{\alpha-2}\mathbb{E}(T_{m,0}^{\nu})+\alpha^2\theta^{2(\alpha-1)}\mathbb{E}((T_{m,0}^{\nu})^2)
	\end{align}
	where we have used the second moment of tempered stable subordinator in the penultimate step (see Gupta {\it et al.} (2020), p. 948). On using the second moment of the first visit time to state zero for BDPC (see Di Crescenzo {\it et al.} (2008)) in \eqref{2expTmo}, we get the required result.
\end{proof}

\subsection{Effective catastrophe of tempered BDPC}
Recall from Section \ref{SecEffctCatastrophe} that a catastrophe which occurs at a non-zero state is called an effective catastrophe. 
We study the first occurrence time of an effective catastrophe $K_{m,0}^{\theta,\alpha,\nu}$ in tempered BDPC with $\mathscr{N}^{\theta,\alpha,\nu}(0)=m\in I$ by considering the time-changed variant of modified birth-death process with catastrophes $\{\mathcal{M}^{\nu}(t)\}_{t\geq0}$ defined in Section \ref{SecEffctCatastrophe}. That is,  $\mathcal{M}^{\theta,\alpha,\nu}(t)=\mathcal{M}^{\nu}(Y_{\theta,\alpha}(t))$, $t\geq0$. Here, the time-changing component considered is the inverse tempered stable subordinator $\{Y_{\theta,\alpha}(t)\}_{t\geq0}$ with $\theta>0$ as its tempering parameter and $0<\alpha<1$ is the stability index such that it is independent of $\{\mathcal{M}^{\nu}(t)\}_{t\geq0}$. 

Let  $q_{m,n}^{\theta,\alpha,\nu}(t)=\mathrm{Pr}\{\mathcal{M}^{\theta,\alpha,\nu}(t)=n|\mathcal{M}^{\theta,\alpha,\nu}(0)=m\}$, $m\in I$, $n\in I\cup\{-1\}$ be the state probabilities of $\{\mathcal{M}^{\theta,\alpha,\nu}(t)\}_{t\geq0}$. Thus,
\begin{equation}\label{Km0temp}
	\mathrm{Pr}\{K_{m,0}^{\theta,\alpha,\nu}>t\}=\int_{t}^{\infty}g_{m,0}^{\theta,\alpha,\nu}(x)\mathrm{d}x=\sum_{n=0}^{\infty}q_{m,n}^{\theta,\alpha,\nu}(t)=1-q_{m,-1}^{\theta,\alpha,\nu}(t),\,m\in I,
\end{equation}
where $g_{m,0}^{\theta,\alpha,\nu}$ is the pdf of $K_{m,0}^{\theta,\alpha,\nu}$.
On taking Laplace transform of
\begin{equation*}
	q_{m,n}^{\theta,\alpha,\nu}(t)=\int_{0}^{\infty}q_{m,n}^{\nu}(y)\mathrm{Pr}\{Y_{\theta,\alpha}(t)\in\mathrm{d}y\}
\end{equation*}
and by using \eqref{temperedPDFLT}, we obtain
\begin{equation}\label{qmntemp}
	\tilde{q}_{m,n}^{\theta,\alpha,\nu}(z)=\frac{\phi_{\theta,\alpha}(z)}{z}\tilde{q}_{m,n}^{\nu}(\phi_{\theta,\alpha}(z)).
\end{equation}
Further, on taking the Laplace transform of \eqref{DE5}, and by using \eqref{CaputotemperedLT} and \eqref{qmntemp}, we get
{\small\begin{align}\label{DE6}
		\left.
		\begin{aligned}
			\frac{\mathrm{d}^{\theta,\alpha}}{\mathrm{d}t^{\theta,\alpha}}q_{m,-1}^{\theta,\alpha,\nu}(t)&=\nu(1-q_{m,-1}^{\theta,\alpha,\nu}(t)-q_{m,0}^{\theta,\alpha,\nu}(t)),\\
			\frac{\mathrm{d}^{\theta,\alpha}}{\mathrm{d}t^{\theta,\alpha}}q_{m,0}^{\theta,\alpha,\nu}(t)&=-\lambda_{0}q_{m,0}^{\theta,\alpha,\nu}(t)+\mu_{1}q_{m,1}^{\theta,\alpha,\nu}(t),\\
			\frac{\mathrm{d}^{\theta,\alpha}}{\mathrm{d}t^{\theta,\alpha}}q_{m,n}^{\theta,\alpha,\nu}(t)&=-(\lambda_{n}+\mu_{n}+\nu)q_{m,n}^{\theta,\alpha,\nu}(t)+\lambda_{n-1}q_{m,n-1}^{\theta,\alpha,\nu}(t)+\mu_{n+1}q_{m,n+1}^{\theta,\alpha,\nu}(t),\,n=1,2,\dots
		\end{aligned}
		\right\}
\end{align}}
with initial condition $q_{m,n}^{\theta,\alpha,\nu}(0)=\delta_{m,n}$.

On using Eq. (11) and Eq. (12) of Di Crescenzo {\it et al.} (2008) in \eqref{qmntemp}, we have the following result:
\begin{theorem}
For $m,n\in I$, the Laplace transforms of state probabilities of $\{\mathcal{M}^{\theta,\alpha,\nu}(t)\}_{t\geq0}$ are
\begin{equation}\label{qm1ztemp}
	\tilde{q}^{\theta,\alpha,\nu}_{m,-1}(z)=\frac{\nu \phi_{\theta,\alpha}(z)}{z(\nu+\phi_{\theta,\alpha}(z))}\Bigg(\frac{1}{\phi_{\theta,\alpha}(z)}-\frac{\psi_{\theta,\alpha}^{\nu}(z)\tilde{\mathrm{p}}_{m,0}^{\theta,\alpha}\big((\nu+(z+\theta)^{\alpha})^{1/\alpha}-\theta\big)}{1-\nu\psi_{\theta,\alpha}^{\nu}(z)\tilde{\mathrm{p}}_{0,0}^{\theta,\alpha}\big((\nu+(z+\theta)^{\alpha})^{1/\alpha}-\theta\big)}\Bigg)
\end{equation}
and
\begin{align*}
	\tilde{q}^{\theta,\alpha,\nu}_{m,n}(z)&=\frac{\phi_{\theta,\alpha}(z)\psi_{\theta,\alpha}^{\nu}(z)}{z}\Bigg(\tilde{\mathrm{p}}_{m,n}^{\theta,\alpha}\big((\nu+(z+\theta)^{\alpha})^{1/\alpha}-\theta\big)\nonumber\\
	&\ \
	+\frac{\nu\psi_{\theta,\alpha}^{\nu}(z)\tilde{\mathrm{p}}_{m,0}^{\theta,\alpha}\big((\nu+(z+\theta)^{\alpha})^{1/\alpha}-\theta\big)\tilde{\mathrm{p}}_{0,n}^{\theta,\alpha}\big((\nu+(z+\theta)^{\alpha})^{1/\alpha}-\theta\big)}{1-\nu\psi_{\theta,\alpha}^{\nu}(z)\tilde{\mathrm{p}}_{0,0}^{\theta,\alpha}\big((\nu+(z+\theta)^{\alpha})^{1/\alpha}-\theta\big)} \Bigg),\,z>0.
\end{align*}
\end{theorem}

The proof of following result is along the similar lines to that of Theorem \ref{thm3.6}. Hence, it is omitted.

\begin{theorem}\label{thmEffctPdfLT}
The pdf of first occurrence time of an effective catastrophe in tempered BPDC has the following Laplace transform:
{\small\begin{equation*}
		\tilde{g}_{m,0}^{\theta,\alpha,\nu}(z)=\frac{\nu \phi_{\theta,\alpha}(z)}{\nu+\phi_{\theta,\alpha}(z)}\Big(\frac{1}{\phi_{\theta,\alpha}(z)}-\frac{\psi_{\theta,\alpha}^{\nu}(z)\tilde{f}_{m,0}({\nu+\phi_{\theta,\alpha}(z)})\tilde{\mathrm{p}}_{0,0}^{\theta,\alpha}\big((\nu+(z+\theta)^{\alpha})^{1/\alpha}-\theta\big)}{1-\nu\psi_{\theta,\alpha}^{\nu}(z) \tilde{\mathrm{p}}_{0,0}^{\theta,\alpha}\big((\nu+(z+\theta)^{\alpha})^{1/\alpha}-\theta\big)}\Big),\,z>0,
\end{equation*}}
where $\mathscr{N}^{\theta,\alpha,\nu}(0)=m\in I$.
\end{theorem}

\begin{corollary}
If zero is an absorbing state of tempered BDPC then
\begin{equation*}
	\tilde{g}_{m,0}^{\theta,\alpha,\nu}(z)=\tilde{f}_{m,0}^{\theta,\alpha,\nu}(z)-\frac{\phi_{\theta,\alpha}(z)\tilde{f}_{m,0}(\nu+\phi_{\theta,\alpha}(z))}{(\nu+\phi_{\theta,\alpha}(z))(1-\nu\psi_{\theta,\alpha}^{\nu}(z)\tilde{\mathrm{p}}_{0,0}^{\theta,\alpha}\big((\nu+(z+\theta)^{\alpha})^{1/\alpha}-\theta\big))}.
\end{equation*}
\end{corollary}

\begin{theorem}
The expectation of first occurrence of an effective catastrophe is given by
\begin{equation*}
	\mathbb{E}(K_{m,0}^{\theta,\alpha,\nu})=\alpha\theta^{\alpha-1}\Big(\frac{1}{\nu}+\frac{\tilde{p}_{m,0}(\nu)}{1-\nu\tilde{p}_{0,0}(\nu)}\Big).
\end{equation*}
\end{theorem}
\begin{proof}
As $K_{m,0}^{\theta,\alpha,\nu}$ is a non-negative continuous random variable, we have 
{\small\begin{align}\label{EffectExpct}
	\mathbb{E}(K_{m,0}^{\theta,\alpha,\nu})
	&=\int_{0}^{\infty}\mathrm{Pr}\{K_{m,0}^{\theta,\alpha,\nu}>t\}\mathrm{d}t\nonumber\\
	&=\int_{0}^{\infty}(1-q_{m,0}^{\theta,\alpha,\nu}(t))\mathrm{d}t\ \ (\text{from \eqref{Km0temp}})\nonumber\\
	&=\int_{0}^{\infty}\int_{0}^{\infty}(1-q^{\nu}_{m,0}(y))\mathrm{Pr}\{Y_{\theta,\alpha}(t)\in\mathrm{d}y\}\mathrm{d}t\nonumber\\
	&=\int_{0}^{\infty}\int_{0}^{\infty}\mathrm{Pr}\{K_{m,0}^{\nu}>y\}\mathrm{Pr}\{Y_{\theta,\alpha}(t)\in\mathrm{d}y\}\mathrm{d}t\ \ (\text{see Di Crescenzo {\it et al.} (2008), Eq. (9)})\nonumber\\
	&=\int_{0}^{\infty}\mathrm{Pr}\{K_{m,0}^{\nu}>Y_{\theta,\alpha}(t)\}\mathrm{d}t\nonumber\\
	&=\int_{0}^{\infty}\mathrm{Pr}\{D_{\theta,\alpha}(K_{m,0}^{\nu})>t\}\mathrm{d}t\nonumber\\ &=\mathbb{E}(D_{\theta,\alpha}(K_{m,0}^{\nu}))\nonumber\\
	&=\int_{0}^{\infty}\mathbb{E}(D_{\theta,\alpha}(x))\mathrm{Pr}\{K_{m,0}^{\nu}\in\mathrm{d}x\}\nonumber\\
	&=\alpha\theta^{\alpha-1}\mathbb{E}(K_{m,0}^{\nu}),
\end{align}}
where we have used the expectation of tempered stable subordinator (see Gupta {\it et al.} (2020), p. 948). On using the expectation of first occurrence time of an effective catastrophe in BDPC $\{\mathcal{N}^{\nu}(t)\}_{t\geq0}$, we get the required result.
\end{proof}

\begin{theorem}
The variance of first occurrence time of an effective catastrophe in tempered BDPC is given by
{\footnotesize\begin{align*}
	\mathrm{Var}(K_{m,0}^{\theta,\alpha,\nu})
	&=\alpha(1-\alpha)\theta^{\alpha-2}\Big(\frac{1}{\nu}+\frac{\tilde{p}_{m,0}(\nu)}{1-\nu\tilde{p}_{0,0}(\nu)}\Big)\\
	&\ \ +\frac{\alpha^{2}\theta^{2\alpha-2}}{\nu^{2}}\Big(1-\Big(\frac{\nu\tilde{p}_{m,0}(\nu)}{1-\nu\tilde{p}_{0,0}(\nu)}\Big)^2
   -\frac{2\nu^2}{1-\nu\tilde{p}_{0,0}(\nu)}\frac{\mathrm{d}}{\mathrm{d}\nu}\tilde{p}_{m,0}(\nu)-\frac{2\nu^3\tilde{p}_{m,0}(\nu)}{(1-\nu\tilde{p}_{0,0}(\nu))^2}\frac{\mathrm{d}}{\mathrm{d}\nu}\tilde{p}_{0,0}(\nu)\Big).
\end{align*}}
\end{theorem}
\begin{proof}
As $K_{m,0}^{\theta,\alpha,\nu}$ is a non-negative continuous random variable, we have
\begin{align}
		\mathbb{E}((K_{m,0}^{\theta,\alpha,\nu})^2)
		&=2\int_{0}^{\infty}t\mathrm{Pr}\{K_{m,0}^{\theta,\alpha,\nu}>t\}\mathrm{d}t\nonumber\\
		&=2\int_{0}^{\infty}t\mathrm{Pr}\{K_{m,0}^{\nu}>Y_{\theta,\alpha}(t)\}\mathrm{d}t\nonumber\\
		&=2\int_{0}^{\infty}t\mathrm{Pr}\{D_{\theta,\alpha}(K_{m,0}^{\nu})>t\}\mathrm{d}t\nonumber\\ &=\mathbb{E}((D_{\theta,\alpha}(K_{m,0}^{\nu}))^2)\nonumber\\
		&=\int_{0}^{\infty}\mathbb{E}((D_{\theta,\alpha}(x))^2)\mathrm{Pr}\{K_{m,0}^{\nu}\in\mathrm{d}x\}\nonumber\\
		&=\alpha(1-\alpha)\theta^{\alpha-2}\mathbb{E}(K_{m,0}^{\nu})+\alpha^2\theta^{2\alpha-2}\mathbb{E}((K_{m,0}^{\nu})^2),
\end{align}
where we have used $\mathbb{E}(D_{\theta,\alpha}(x))=\alpha\theta^{\alpha-1}x$ and $\mathrm{Var}(D_{\theta,\alpha}(x))=\alpha(1-\alpha)\theta^{\alpha-2}x$ (see Gupta {\it et al.} (2020), p. 948). Thus, we have the required result.
\end{proof}

\subsection{Tempered linear birth-death process with catastrophe}
Here, we study a tempered variant of the linear birth-death process with catastrophe, namely, the tempered LBDPC. The tempered BDPC reduces to the tempered LBDPC upon setting $\lambda_n=n\lambda$ and $\mu_{n}=n\mu$ for $n\in I= \{0,1,2,\dots\}$. Consequently, the state zero becomes an absorbing state.

For $m,n\in I$, let $\mathbf{p}_{m,n}^{\theta,\alpha,\nu}(t)=\mathrm{Pr}\{\mathbf{N}^{\theta,\alpha,\nu}(t)=n|\mathbf{N}^{\theta,\alpha,\nu}(0)=m\}$ be the state probabilities of tempered LBDPC $\{\mathbf{N}^{\theta,\alpha,\nu}(t)\}_{t\geq0}$. Also, let $\{\mathbf{N}^{\theta,\alpha}(t)\}_{t\geq0}$ be the corresponding tempered linear birth-death process without catastrophe, abbreviated as the tempered LBDP. We denote its state probabilities by
$\mathbf{p}_{m,n}^{\theta,\alpha}(t)=\mathrm{Pr}\{\mathbf{N}^{\theta,\alpha}(t)=n|\mathbf{N}^{\theta,\alpha}(0)=m\}$. 

The governing system of differential equations for the state probabilities of tempered LBDPC is given by
{\footnotesize\begin{align}\label{DE7temp}
		\left.
		\begin{aligned}
			\frac{\mathrm{d}^{\theta,\alpha}}{\mathrm{d}t^{\theta,\alpha}}\mathbf{p}_{m,0}^{\theta,\alpha,\nu}(t)(t)
			&=-\nu \mathbf{p}_{m,0}^{\theta,\alpha,\nu}(t)+\mu \mathbf{p}_{m,1}^{\theta,\alpha,\nu}(t)+\nu,\\
			\frac{\mathrm{d}^{\theta,\alpha}}{\mathrm{d}t^{\theta,\alpha}}\mathbf{p}_{m,n}^{\theta,\alpha,\nu}(t)
			&=-(n\lambda+n\mu+\nu)\mathbf{p}_{m,n}^{\theta,\alpha,\nu}(t)+(n-1)\lambda \mathbf{p}_{m,n-1}^{\theta,\alpha,\nu}(t) +(n+1)\mu \mathbf{p}_{m,n+1}^{\theta,\alpha,\nu}(t),\,n\geq1
		\end{aligned}
		\right\}
\end{align}}
with initial condition $\mathbf{p}_{m,n}^{\theta,\alpha,\nu}(0)=\delta_{m,n}$.

 As zero is an absorbing state in the tempered LBDPC, we have $\mathbf{p}_{0,0}^{\theta,\alpha,\nu}(t)=1$. So, $\mathbf{p}_{0,n}^{\theta,\alpha,\nu}(t)=0$, $n\geq1$. 
On using the explicit expressions of the state probabilities of LBDP (see Feller (1968)) in Proposition \ref{thm5.2}, we derive the following results:

These are stated without proofs as the proofs of these results follow along the lines to that of similar results in the previous sections.

For $\lambda=\mu$, the extinction probability of tempered LBDPC is given by
{\small\begin{align*}
		\mathbf{p}_{1,0}^{\theta,\alpha,\nu}(t)
		&=\lambda e^{-\theta t}\sum_{m=0}^{\infty}\theta^{m}
		\int_{0}^{\infty}e^{-x}
		\Big(t^{\alpha+m}E_{\alpha,\alpha+m+1}^{2}((\theta^{\alpha}-\lambda x-\nu)t^{\alpha})\nonumber\\
		&\ \
		-\theta^{\alpha}t^{2\alpha+m}E_{\alpha,2\alpha+m+1}^{2}((\theta^{\alpha}-\lambda x-\nu)t^{\alpha})
		\Big)\mathrm{d}x +\nu e^{-\theta t}\sum_{n=0}^{\infty}(\theta t)^n t^{\alpha}E_{\alpha,\alpha+n+1}((\theta^{\alpha}-\nu)t^{\alpha}),
\end{align*}}
for $\lambda>\mu$, it is given by
{\small\begin{align*}
		\mathbf{p}_{1,0}^{\theta,\alpha,\nu}(t)
		&=\frac{\mu e^{-\theta t}}{\lambda}\sum_{m=0}^{\infty}\sum_{k=0}^{\infty}(-\nu t^{\alpha})^m(\theta t)^k E_{\alpha,m\alpha+k+1}^{m}(\theta^{\alpha}t^{\alpha})-e^{-\theta t}\Big(\frac{\lambda-\mu}{\lambda}\Big)\sum_{r=1}^{\infty}\sum_{m=0}^{\infty}\sum_{k=0}^{\infty}\Big(\frac{\mu}{\lambda}\Big)^r\nonumber\\
		&\ \ \cdot(-(\nu+(\lambda-\mu)r)t^{\alpha})^m(\theta t)^k E_{\alpha,m\alpha+k+1}^m(\theta^{\alpha}t^{\alpha})+\nu e^{-\theta t}\sum_{n=0}^{\infty}(\theta t)^n t^{\alpha}E_{\alpha,\alpha+n+1}((\theta^{\alpha}-\nu)t^{\alpha}),
\end{align*}}
and for $\lambda<\mu$, it has the following form:
{\small\begin{align*}
		\mathbf{p}_{1,0}^{\theta,\alpha,\nu}(t)
		&=e^{-\theta t}\sum_{m=0}^{\infty}\sum_{k=0}^{\infty}(-\nu t^{\alpha})^m(\theta t)^k E_{\alpha,m\alpha+k+1}^{m}(\theta^{\alpha}t^{\alpha})-e^{-\theta t}\Big(\frac{\mu-\lambda}{\lambda}\Big)\sum_{r=1}^{\infty}\sum_{m=0}^{\infty}\sum_{k=0}^{\infty}\Big(\frac{\lambda}{\mu}\Big)^r\nonumber\\
		&\ \ \cdot(-(\nu+(\mu-\lambda)r)t^{\alpha})^m(\theta t)^k E_{\alpha,m\alpha+k+1}^m(\theta^{\alpha}t^{\alpha})+\nu e^{-\theta t}\sum_{n=0}^{\infty}(\theta t)^n t^{\alpha}E_{\alpha,\alpha+n+1}((\theta^{\alpha}-\nu)t^{\alpha}).
\end{align*}}

Similarly, for any $n\geq1$, the state probability $\mathbf{p}_{1,n}^{\theta,\alpha,\nu}(t)$ of tempered LBDPC is given as follows:

For the case $\lambda=\mu$, we have
{\small\begin{align*}
	\mathbf{p}_{1,n}^{\theta,\alpha,\nu}(t)
	&=\frac{(-\lambda)^{n-1}}{n!}e^{-\theta t}\sum_{m=0}^{\infty}\sum_{k=0}^{\infty}\frac{\mathrm{d}^n}{\mathrm{d}\lambda^n}
	\Big(\lambda\int_{0}^{\infty}e^{-x}((\theta^{\alpha}-\lambda x-\nu)t^{\alpha})^{m}(\theta t)^{k}E_{\alpha,m\alpha+k+1}^{m}(\theta^{\alpha}t^{\alpha})\mathrm{d}x
	\Big),
\end{align*}	}
for $\lambda>\mu$, it is given by
\begin{align*}
	\mathbf{p}_{1,n}^{\theta,\alpha,\nu}(t)
	&=e^{-\theta t}\Big(\frac{\lambda-\mu}{\lambda}\Big)^2\sum_{l=0}^{\infty}\sum_{m=0}^{\infty}\sum_{k=0}^{\infty}\sum_{r=0}^{n-1}(-1)^r \Big(\frac{\mu}{\lambda}\Big)^l \binom{n+l}{l}\binom{n-1}{r}\nonumber\\
	&\hspace{3cm} \cdot(-(\nu+(\lambda-\mu)(l+r+1))t^{\alpha})^m (\theta t)^k E_{\alpha,m\alpha+k+1}^m(\theta^{\alpha}t^{\alpha}),
\end{align*}
and for $\lambda<\mu$, it has the following form:
\begin{align*}
	\mathbf{p}_{1,n}^{\theta,\alpha,\nu}(t)
	&=e^{-\theta t}\Big(\frac{\lambda}{\mu}\Big)^{n-1}\Big(\frac{\mu-\lambda}{\mu}\Big)^2\sum_{l=0}^{\infty}\sum_{m=0}^{\infty}\sum_{k=0}^{\infty}\sum_{r=0}^{n-1}(-1)^r \Big(\frac{\lambda}{\mu}\Big)^l \binom{n+l}{l}\binom{n-1}{r}\nonumber\\
	&\hspace{4cm} \cdot(-(\nu+(\mu-\lambda)(l+r+1))t^{\alpha})^m (\theta t)^k E_{\alpha,m\alpha+k+1}^m(\theta^{\alpha}t^{\alpha}).
\end{align*}

For the tempered LBDPC, the probability generating function is given by
\begin{equation}\label{pgftempp}
	G_{m}^{\theta,\alpha,\nu}(x,t)=\sum_{n=0}^{\infty}x^{n}\mathbf{p}_{m,n}^{\theta,\alpha,\nu}(t),\,t\geq0,\,|x|\leq1.
\end{equation}
On taking the Caputo tempered fractional derivative on both sides of \eqref{pgftempp} and using \eqref{DE7temp}, we get
\begin{equation}\label{pgfDEtemp}
	\frac{\partial^{\theta,\alpha}}{\partial t^{\theta,\alpha}}G_{m}^{\theta,\alpha,\nu}(x,t)=\frac{\partial}{\partial x}G_{m}^{\theta,\alpha,\nu}(x,t)(\lambda x-\mu)(x-1)+\nu(1-G_{m}^{\theta,\alpha,\nu}(x,t)),
\end{equation}
with initial condition $G_{m}^{\theta,\alpha,\nu}(x,0)=x^{m}$. 

From \eqref{pgfDEtemp}, we obtain the expectation of tempered LBDPC in the following form:
\begin{equation*}
	\mathbb{E}(\mathbf{N}^{\theta,\alpha,\nu}(t))=e^{-\theta t}\sum_{m=0}^{\infty}\sum_{k=0}^{\infty}((\lambda-\mu-\nu)t^{\alpha})^m (\theta t)^k E_{\alpha,m\alpha+k+1}^{m}(\theta^{\alpha} t^{\alpha}),
\end{equation*}
and its variance, for $\lambda\neq\mu$, is given by
{\small\begin{align*}
	Var(\mathbf{N}^{\theta,\alpha,\nu}(t))
	&=\frac{2\lambda}{\lambda-\mu}e^{-\theta t}\sum_{m=0}^{\infty}\sum_{k=0}^{\infty}((2\lambda-2\mu-\nu)t^{\alpha})^m (\theta t)^k E_{\alpha,m\alpha+k+1}^{m}(\theta^{\alpha} t^{\alpha})\nonumber\\
	&\ \ - \Big(\frac{\lambda+\mu}{\lambda-\mu}\Big)e^{-\theta t}\sum_{m=0}^{\infty}\sum_{k=0}^{\infty}((\lambda-\mu-\nu)t^{\alpha})^m (\theta t)^k E_{\alpha,m\alpha+k+1}^{m}(\theta^{\alpha} t^{\alpha}) 
	- \big(\mathbb{E}(\mathbf{N}^{\theta,\alpha,\nu}(t))\big)^2
\end{align*}}
and for $\lambda=\mu$, it is given by
\begin{align*}
	Var(\mathbf{N}^{\theta,\alpha,\nu}(t))
	&=2\lambda e^{-\theta t}\sum_{r=0}^{\infty}(\theta t)^{r}\Big(t^{\alpha}E_{\alpha,\alpha+r+1}^{2}((\theta^{\alpha}-\nu)t^{\alpha})-\theta^{\alpha}t^{2\alpha}E_{\alpha,2\alpha+r+1}^{2}((\theta^{\alpha}-\nu)t^{\alpha})	\Big)\nonumber\\
	& \hspace{2.3cm} +e^{-\theta t}\sum_{m=0}^{\infty}\sum_{k=0}^{\infty}(-\nu t^{\alpha})^{m}(\theta t)^{k}E_{\alpha,m\alpha+k+1}^{m}(\theta^{\alpha}t^{\alpha})-\big(\mathbb{E}(\mathbf{N}^{\theta,\alpha,\nu}(t))\big)^2,
\end{align*}
where we have taken $\mathbf{p}_{1,1}^{\theta,\alpha,\nu}(0)=1$.

\section*{Acknowledgement}
The research of second author was supported by a UGC fellowship, NTA reference no. 231610158041, Govt. of India.

\end{document}